\documentclass[12pt,twoside,reqno]{amsart}
\usepackage[colorlinks=true,citecolor=blue]{hyperref}
\usepackage{mathptmx, amsmath, amssymb, amsfonts, amsthm, mathptmx, enumerate, color}
\usepackage{graphicx}
\usepackage{epstopdf}

\newtheorem{thm}{Theorem}[section]
\newtheorem{cor}{Corollary}[section]
\newtheorem{lem}{Lemma}[section]
\newtheorem{pro}{Proposition}[section]
\theoremstyle{definition}
\newtheorem{defi}{Definition}[section]
\newtheorem{exam}{Example}[section]
\newtheorem{remark}{Remark}[section]

\numberwithin{equation}{section}

\def\epi{\mbox{\rm epi}\,}

\def\gph{\mbox{\rm gph}\,}
\def\dom{\mbox{\rm dom}\,}

\def \N{\mathbb{N}}
\def \R{\mathbb{R}}
\def \B{\mathbb{B}}

\begin{document}
\setcounter{page}{1}

\vspace*{1.0cm}
\title[Non-subdifferentiability Optimality  and mean value theorems]
{Non-subdifferentiability Optimality  and mean value theorems via new relative subdifferentials}
\author[V.D. Thinh, T.D. Chuong, X. Qin]{ Vo Duc Thinh$^{1,2}$, Thai Doan Chuong$^{3}$, Xiaolong Qin$^{1,4,*}$}
\maketitle
\vspace*{-0.6cm}

\begin{center}
{\footnotesize {\it

$^1$School of Mathematical Sciences, Zhejiang Normal University, Jinhua 321004, China\\
$^2$Dong Thap University, Cao Lanh 870000, Dong Thap, Vietnam\\
$^3$Department of Mathematics, Brunel University London, London, England\\
$^4$Nanjing Center for Applied Mathematics, Nanjing, China\\
}
}
\end{center}

\vskip 4mm {\small\noindent {\bf Abstract.}
Motivated by the optimality principles for non-subdifferentiable optimization problems, we  introduce new relative subdifferentials and examine some properties for relatively lower semicontinuous functions including  $\epsilon$-regular subdifferential and  limiting subdifferential relative to a set. The fuzzy sum rule for the relative $\epsilon$-regular subdifferentials and the sum rule for the relative limiting subdifferentials  are established. We utilize these relative subdifferentials to establish optimality conditions for non-subdifferentiable optimization problems under  mild constraint qualifications. Examples are given to demonstrate that the  optimality conditions obtained work better and sharper than some existing results. We also provide different versions of mean value theorems via the relative subdifferentials and employ them to characterize the equivalences between the convexity relative to a set and the monotonicity of the relative subdifferentials of a non-subdifferentiable function.


}

\renewcommand{\thefootnote}{}
\footnotetext{ $^*$Corresponding author.
\par
E-mail address:
thinhvd.k4@gamil.com; vdthinh@dthu.edu.vn (V.D. Thinh), chuong.thaidoan@brunel.ac.uk (T.D. Chuong), qxlxajh@163.com (X. Qin).
 }

\section{Introduction}\label{intro}
When it comes to dealing with non-differential functions and nonsmooth optimization problems,  generalized differentials and subdifferentials  always play a crucial role  in  analyzing and  establishing foundational calculus rules and tools for variational analysis and  nonsmooth optimization theory. Basically, each type of generalized subdifferentials or differentials is  well-defined and suitable for a certain class of non-differential functions, such as the directional derivative and convex subdifferential for the convex functions \cite{Roc96}, the Clarke  derivative and subdifferential for the locally Lipschitz functions \cite{CLSW98,Clarke1983}, and the basic subdifferential for the regular functions \cite{Mor06,Mor2018}. The generalized subdifferentials and  differentials including the convex subdifferentials, derivatives  and coderivatives are efficient tools not only  to establish necessary and sufficient optimality conditions for nonsmooth optimization problems but also to characterize the stability properties for functions  and set-valued mappings; see, e.g., \cite{CLSW98,Clarke1983, Mor93, Mor94, Mor06,Mor2018, Roc96} and the references therein.

A related research aspect when studying the subdifferentials, which has also received much attention from researchers,  is the establishment of {\it mean value/Lagrange-type} theorems  by using different types of generalized subdifferentials and differentials. For examples, the authors in \cite{Hir1980, Wegge1974} established  mean value theorems for the convex functions and the authors in \cite{Clarke1983,CL1994,Lebourg1975,Zaorodny1988,Zaorodny1990} provided versions of the mean value theorems for the class of Lipschitz functions by way of the Clarke subdifferential. For the class of lower semicontinuous functions, the authors in \cite{Penot1997,Penot1988}  stated mean value theorems by adopting lower derivative and lower contingent subdifferential, while the authors \cite{ACL1995} examined mean value theorems by using an abstract subdifferential. Other versions of mean value theorems for the lower semicontinuous functions through the Fr\'{e}chet subdifferential    can be found in \cite{Mor1988,Mor06,Mor2018,Trang2012}.
Using a mean value theorem as an efficient tool to establish  equivalent characterizations for Asplund spaces was given in \cite{Trang2012}.
The interested reader  is referred to \cite{ACL1995,Mor1988,Mor06,Mor2018,PS1997,TZ1995} for the  applications of mean value theorems to characterize generalized convex functions.

Let us mention some recent developments of generalized subdifferentials. The authors in \cite{MWY2023b,MWY2023a} presented some types of coderivatives relative to a set, corresponding sum rules and relatively  extremal principle. They  used them as a tool to characterize necessary and sufficient conditions for the Aubin property/ Lipschitz-like and metric regularity relative to a set  of multifunctions and the solution mapping of variational inequality systems.  Independently, the authors in \cite{TQY01-ZJNU,TQY02-ZJNU} introduced different versions of coderivatives and subdifferentials with relative to a set of multifunctions and applied them to provide equivalent characterizations for Aubin property/Lipschitz-like and metric regularity relative to a set of multifunctions as well as the locally Lipschitz continuity and metric regularity of functions. Moreover, optimality conditions via the proposed generalization for specific optimization problems were also  established in those papers.

In this paper, we mainly focus on the following non-subdifferentiable sum optimization problem:
\begin{equation}\label{SP} \tag{SP}
	\min \{f_1(x)+f_2(x)\mid  x\in \mathcal C\},
\end{equation} where  $\mathcal C$ is a nonempty closed and convex set in a Banach space $X$, $f_1: X\to\bar\R:=\R\cup\{\pm \infty\}$ is locally Lipschitz  relative to  $\mathcal C$,  and  $f_2: X\to\bar\R$ is lower semicontinuous relative to $\mathcal C$ around the reference point (cf.~their definitions  in Section~2). It is worth noticing that, in a special case where $f_1$ is a convex function and $f_2$ is a concave function,  problem~\eqref{SP} collapses  to a popular class of DC optimization problems  (see, e.g., \cite{DNV2010,HT1999,DCtao-an-97}).

Motivating by optimality conditions for   non-subdifferentiable optimization problem~\eqref{SP}, we first  introduce  new types of relative subdifferentials called   {\it $\epsilon$-regular subdifferential} and {\it  limiting subdifferential relative to a set} for relatively lower semicontinuous functions and explore their properties. We then employ  these relative subdifferentials to establish optimality conditions for the non-subdifferentiable optimization problem~\eqref{SP} under  weak constraint qualifications. We also give examples  to illustrate  that the optimality conditions obtained  in this paper work better and sharper than some existing results; see, e.g., \cite[Theorem~8.15]{Roc98} and \cite[Proposition~5.3]{Mor06}. In this vein, we examine different versions of mean value theorems via the relative subdifferentials and utilize   them to  characterize the equivalences  between the convexity relative to a set and the monotonicity of the relative subdifferentials of the non-subdifferentiable functions. In addition, we present the fuzzy sum rule for $\epsilon$-regular subdifferentials and the sum rule for the limiting subdifferentials relative to a set.

The rest of the paper is organized as follows. In Section~2, we recall notations and basic results from variational analysis. Section 3 is devoted to   introducing  the concepts of relative subdifferentials and providing their calculus rules.  In Section~4, we establish optimality conditions for problem~\eqref{SP}.  We also present exact and approximate sum rules for the relative subdifferentials. Section~5  provides some versions of mean value theorems via the relative subdifferentials. In this section, we also  establish the equivalence of the convex property on a set of a function  and the monotonicity of subdifferentials relative to a set. The last section, Section~6, summarizes the main results and provides research perspectives.

\section{Preliminaries and Tools of Variational Analysis}

Throughout this paper, the notation $\N$ and $\R$ are the natural numbers set and the real numbers set respectively, while $\bar{\R}:=\R\cup\{\pm\infty\}$ (with convention $\infty-\infty=\infty$) and $\R_+:=[0,\infty).$ Consider a Banach space $X$ and its topological dual $X^*.$ We say that $X$ is reflexive if $X$ coincides under the canonical imbedding with $X^{**}:=(X^*)^*$. We denote by $\Vert \cdot\Vert$ and $\Vert \cdot\Vert_*$ the norms of $X$ and $X^*$ respectively. A closed  ball in a space  centered at $z$ with radius $r>0$  is denoted by $\B(z,r)$.  The closed unit balls of $X$ and $X^*$ are denoted by $\B_X$ and $\B_{X^*}$ respectively, while $\B:=[-1,1]\subset \R.$ For a pair $(x,x^*)\in X\times X^*,$ the symbol $\langle x^*, x\rangle$ indicates the canonical pairing between $X$ and $X^*.$ The notation $x_k\to x$ and $x_k\rightharpoonup x$ denote the strong convergence and weak convergence of a sequence $\{x_k\}$  to $x$, respectively, where $k\in \N:=\{1,2,...\}$. Denote $x_k^*\rightharpoonup^* x^*$ for the weak$^*$ convergence of a sequence $\{x_k^*\}$  to $x^*$  in $X^*.$ Given $\Omega\subset X,$ we use the notion $x_k\xrightarrow{\Omega}\bar x$ to say that $x_k\to \bar x$ and $x_k\in \Omega$ for all $k\in \N$ while $x_k\xrightarrow[f]{\Omega}\bar x$ means that $x_k\xrightarrow{\Omega}\bar x$ and $f(x_k)\to f(\bar x)$. Let $(t_k)$ be a real number sequence. We denote $t_k\downarrow 0$ if $t_k\to 0$ and $t_k\ge t_{k+1}> 0$ for all $k\in \N,$ while the notion $t_k\to 0^+$ means that $t_k\to 0$ and $t_k> 0$ for all $k\in \N.$ The metric projection of a set $\Omega\subset X\times Y$ onto $X$ is signified by $\Pi_{X}(\Omega).$

A set-valued mapping $F: X\rightrightarrows X^*$ is called {\it monotone} if
$$\langle x^*-u^*, x-u\rangle\ge \Vert x-u\Vert\; \text{ \rm for all } x,u\in X \text{ \rm and } x^*\in F(x), y^*\in F(y).$$ The mapping $J: X\rightrightarrows X^*$ defined by $$J(x):=\left\{x^*\in X^*\mid \langle x^*, x\rangle = \Vert x\Vert^2=\Vert x^*\Vert_*^2\right\},\; \, x\in X,$$
is called the {\it dual mapping} from $X$ to $X^*.$ It is known that $J(\lambda x)=\lambda J(x)$ for any $x\in X$ and $\lambda\in \R.$ Note that $J(x)\ne \emptyset$ for all $x\in X$ due to \cite[Proposition~2.12]{Bonnans2000}.

From now on, we assume unless otherwise stated that $X$ is a reflexive Banach space, which is equipped a norm $\Vert \cdot\Vert$, such that both spaces $X$ and $X^*$ are locally uniformly convex and the norm $\Vert \cdot\Vert$ is Fr\'{e}chet differentiable at any nonzero point of $X$. In this case, we have the following properties:
\begin{itemize}
    \item the weak convergence coincides with the weak$^*$ convergence.
    \item the unit ball is weakly sequential compact on $X^*$, i.e., for any sequence $\{b_k\}\subset \B_{X^*},$ there exists a subsequence $\{b_{k_s}\}$ of $\{b_k\}$ with $b_{k_s}\rightharpoonup b$ for some $b\in \B_{X^*}.$ Consequently, for every bounded sequence on $X^*$, there is a weak convergence subsequence on $X^*$.
    \item the norm function is weakly lower semicontinuous on $X^*$, i.e., for any sequence $x_k^*\rightharpoonup  x^*$ with $x_k^*, x^*\in X^*$ for all $k\in \N$, we have $\liminf_{k\to\infty}\Vert x_k^*\Vert\ge \Vert x^*\Vert$.
    \item the duality mapping $J$ is single-valued and bicontinuous.
\end{itemize}

\begin{defi} \label{tangent-cone} {\rm Let $\Omega$  be a nonempty and locally closed  set in a Banach space $X$.

    {\rm (i)} \cite[Definition~6.1]{Roc98} A vector $v\in X$ is a {\it tangent vector} to the set $\Omega$ at $\bar x\in \Omega$, denoted by $v\in T(\bar x, \Omega),$ if there exist sequences $t_k\to 0^+$ and $v_k\to v$ such that $\bar x+t_kv_k\in \Omega$ for all $k\in \N$, where  $T(\bar x, \Omega)$ is called the {\it tangent cone} to $\Omega$ at $\bar x.$

    {\rm (ii)}  \cite[Definition~1.1~(i)]{Mor06} Given $\epsilon\ge 0$,  the $\epsilon$-{\it normal set} to $\Omega$ at $\bar x\in \Omega$, denoted by $\widehat N_{\epsilon}(\bar x,\Omega)$, is defined by
    \begin{equation*}
        \label{e-normal}
        \widehat N_{\epsilon}(\bar x,\Omega)=\left\{x^*\in X^*\mid \limsup_{x\xrightarrow{\Omega}\bar x}\dfrac{\langle x^*, x-\bar x\rangle}{\Vert x-\bar x\Vert}\le \epsilon\right\}.
    \end{equation*}
    When $\epsilon=0,$ $\widehat N_{0}(\bar x,\Omega):=\widehat N(\bar x,\Omega)$ is called the regular/Fr\'echet normal cone to $\Omega$ at $\bar x$.
    }
\end{defi}

\begin{remark}
It is known from \cite[page 6]{Mor06} that
\begin{equation}\label{rem21-eq1}
\widehat{N}_{\epsilon}(\bar x,\Omega)\supset \widehat{N}(\bar x,\Omega)+\epsilon\B_{X^*}
\end{equation}
for any $\epsilon\ge 0$ and $\bar x\in \Omega\subset X.$  Inclusion \eqref{rem21-eq1} becomes the equality if $\Omega$ is convex. In this case, we have $\widehat{N}_{\epsilon}(\bar x,\Omega)=\left\{x^*\in X^*\mid \langle x^*,x-\bar x\rangle\le \epsilon\Vert x-\bar x\Vert\; \forall x\in \Omega\right\}.$
\end{remark}

Given a function $f: X\to \bar{\R},$ the domain and epigraph of $f$ are respectively denoted by
$$\dom f:=\left\{ x\in X\mid f(x)<\infty\right\}\;  \text{ \rm and }\; \epi f:=\left\{(x,y)\in X\times\R\mid y\ge f(x)\right\}.$$ The function $f$ is {\it proper} if $f(x)>-\infty$ for all $x\in X$ and $\dom f\ne \emptyset$, and it is {\it lower semicontinuous (lsc)  at} $\bar x\in \dom f$ if $\liminf_{x\to \bar x}f(x)\ge f(\bar x).$  Function $f$ is called {\it lsc} if $f$ is lsc at any $x\in \dom f$, while  $f$ is {\it lsc around} $\bar x$ if it is lsc at any $x$ in a neighborhood of $\bar x.$ We say that $f$ is {\it lsc relative to a set $\Omega\subset X$ at} $\bar x\in \Omega\cap\dom f$ if  $\liminf_{x\xrightarrow{\Omega}\bar x}f(x)\ge f(\bar x).$  Similarly,   $f$ is called {\it lsc relative to $\Omega$} if it is  lsc relative to $ \Omega$ at  any $x\in  \Omega\cap \dom f$, and   $f$ is  {\it lsc relative to $\Omega$ around} $\bar x$ if there exists a neighborhood $U$ of $\bar x$ such that $f$ is lsc relative to $ \Omega$ at any $x\in U\cap \Omega$. Then $f$  is called {\it convex on} $\Omega\subset \dom f$ if the following inequality holds:
$f(\lambda x+(1-\lambda)y)\le \lambda f(x) +(1-\lambda) f(y)$ for all $x,y\in  \Omega$ and $\lambda\in [0,1].$ It is easy to see that $f$ is convex on $ \Omega\subset \dom f$ if and only if $f_{\Omega}$ is convex. We say that $f$ is {\it continuous relative to} $\Omega$ if $f$ is finite on $\Omega$ and $\lim_{u\xrightarrow{\Omega}x}f(u)=f(x)$  for any $x\in \Omega$. It is clear that if $f$ is continuous relative to $\Omega$ then it is lsc relative to $\Omega.$

\begin{defi}  \label{lips-vec}
	{\rm (i) Let $\Omega\subset X$ be a nonempty and locally closed, and consider  $f:X\to\bar\R$ and $\bar x\in \Omega\cap \dom f.$   $f$ is said  to be {\it locally Lipschitz  relative to} $\Omega$ around $\bar x$ with a modulus $\ell\ge 0$ if there exists  a neighborhood $U$ of $\bar x$ such that
		\begin{equation}\label{lips-con}
			\vert f(x)-f(u)\vert\le \ell\Vert x-u\Vert,\quad \forall x,u\in U\cap \Omega.
		\end{equation} We denote by ${\rm lip}_{\Omega}f(\bar x)$ the infimum of all of $\ell\ge 0$ satisfying \eqref{lips-con}.\\
	(ii) \cite[ Definition~3.1]{MWY2023b}
	Given  $\Omega\subset X$,  a set-valued map $F:X\rightrightarrows Y$  is said  to have the {\it Aubin property relative to} $\Omega$ around $(\bar x,\bar y)\in \gph F:=\left\{(x,y)\in X\times Y\mid y\in F(x)\right\}$ with a modulus $\ell\ge 0$ if $\gph F$ is locally closed around $(\bar x,\bar y)$ and there exist neighborhoods $U$ of $\bar x$ and  $V$ of $\bar y$ such that
		\begin{equation*}
			F(u)\cap V\subset F(x) +\ell\Vert u-x\Vert \mathbb B_Y,\quad  \forall x,u\in \Omega\cap U.
		\end{equation*}
	}
\end{defi}

\begin{defi} {\rm (i) \cite[Definition~6.1]{MWY2023a} Let $\Omega_i\subset X\times Y, i=1,2$ and $C\subset X$ be nonempty. A pair $(\bar x,\bar y)\in \Omega_1\cap\Omega_2$ is called a {\it local extremal point} of $\{\Omega_1,\Omega_2\}$ {\it relative to} $C$ if $\Pi_{X}(\Omega_1\cup\Omega_2)\subset C$ and there exist a sequence $\{b_k\}$ in $Y$ with $b_k\to 0$ and  a neighborhood $U$ of $(\bar x,\bar y)$ such that $$[\Omega_1+(0,b_k)]\cap\Omega_2\cap U=\emptyset\; \text{ \rm for all } k\in \N.$$
(ii) \cite[Definition~1.83]{Mor06}    \label{T03-def-subgradients}  Let $\epsilon\ge 0$ and
      $f: X \rightarrow \bar{\mathbb{R}}$ be finite at $\bar x\in X$. The {\it geometric $\epsilon$-subdifferential} of $f$ at $\bar{x}$ is defined by
\begin{align*}
    \widehat\partial_{\epsilon}f(\bar x):=\left\{x^*\in X^*\mid (x^*,-1)\in \widehat N_{\epsilon}((\bar x,f(\bar x)),\epi f)\right\}.
\end{align*}
 We put $\widehat\partial_{\epsilon}f(\bar x):=\emptyset$ if $|f(\bar x)|=\infty.$
 }
\end{defi}

\section{Relative Subdifferentials and Their Properties}

This section is devoted to introducing  the concepts of {\it relative subdifferentials} and examining some their properties. In what follows, we use the notation $$f_{\Omega}(x):=\begin{cases}
	f(x)&\text{ \rm if } x\in \Omega,\\
	\infty&\text{ \rm otherwise},
\end{cases}$$ where $\Omega\subset X$ is a nonempty set  and  $f: X\to\bar{\R}$ is a function.
\begin{defi}[Relative subdifferentials]
    \label{T03-def-2}
    {\rm Let  $f: X \rightarrow \bar{\mathbb{R}}$ be finite at $\bar{x}\in \Omega$  and $\Omega\subset X$ be locally closed around $\bar{x}$.

     {\rm (i)} Given $\epsilon\ge 0$, the {\it $\epsilon$-regular subdifferential relative to} $\Omega$ of $f$  at $\bar{x}$   is defined by
\begin{align}\label{prop1-eq1}
    \widehat{\partial}^{\epsilon}_{\Omega}f(\bar x):=\widehat{\partial}_{\epsilon}f_{\Omega}(\bar x)\cap J(T(\bar x,\Omega)).
\end{align}
The $0$-regular subdifferential relative to $\Omega$  of $f$ at $\bar x$ is simply called {\it regular subdifferential relative to} $\Omega$ of $f$ at $\bar x$,  and it is denoted by $\widehat{\partial}_{\Omega}f(\bar x).$

{\rm (ii)} The {\it limiting subdifferential relative to} $\Omega$ of $f$ at $\bar x$ is defined by
\begin{align}\label{prop1-eq2}
    \partial_{\Omega}f(\bar x):=\left\{x^*\in X^*\mid\exists \epsilon_k\downarrow 0, x_k\xrightarrow[f]{\Omega}\bar x, x_k^*\rightharpoonup^* x^* \text{ \rm with } x_k^*\in \widehat{\partial}^{\epsilon_k}_{\Omega}f(x_k)\; \forall k\in \N\right\}.
\end{align}

If $\bar x\notin \Omega$ or $|f(\bar x)|=\infty,$ we put $ \widehat{\partial}^{\epsilon}_{\Omega}f(\bar x):=\partial_{\Omega}f(\bar x)=\emptyset$ for any $\epsilon\ge 0.$
}
\end{defi}

\begin{remark}\label{rem31}
(i) In the current setting, $X$ is assumed to be a reflexive Banach space, so  the weak and weak* convergence  are equivalent. In this case,  the weak* convergent sequence $\{x_k^*\}$ in \eqref{prop1-eq2} can be replaced by a weak convergent sequence to $x^*.$ \\
(ii) By definition, we always have \begin{align}\label{sub-relation}
	\widehat{\partial}_{\Omega}f(x)\subset \partial_{\Omega}f(x)\subset \partial f_{\Omega}(x)
\end{align} and  $\widehat{\partial}_{\Omega}^{\epsilon}(f+c)(x)=\widehat{\partial}_{\Omega}^{\epsilon}f(x)$ for all $x\in X$, where  $c\in\R$ and $\epsilon\ge 0$.
\end{remark}

We now demonstrate  how one can  employ the relative  subdifferentials to  establish  Fermat-type  optimality conditions for the following constrained optimization problem (see, e.g.,  \cite{Mor06,Roc98}):
\begin{equation}\label{prob1}
    \min \{ f(x)\mid x\in \Omega\},
\end{equation}
where $\Omega \subset X$ is a nonempty and closed set and $f: X\to\bar{\R}$ is a lsc function relative to $\Omega$ at the reference point.

\begin{pro}[Fermat-type optimality]\label{thm2}
	Let $\bar x\in\Omega$ be a local optimal solution of problem~\eqref{prob1}. Then,
	\begin{equation}
		\label{prob1-eq1}
		0\in\widehat{\partial}^{\epsilon}_{\Omega}f(\bar x) \text{ \rm for all } \epsilon\ge 0,
	\end{equation} and consequently, \begin{equation}
	\label{prob1-eq1-limit}
	0\in\partial_{\Omega}f(\bar x).
\end{equation}
\end{pro}

\begin{proof}
Since $\bar x$ is a local optimal solution to problem~\eqref{prob1}, one sees that there exists $\delta>0$ such that
$f(x)\ge f(\bar x)$  for all  $x\in \Omega\cap \B(\bar x,\delta).$ For any $(x, r)\in \epi f_{\Omega}\cap\B((\bar x, f(\bar x)),\delta)$, it holds that $\Vert x-\bar x\Vert\le \delta$ and $r\ge f_{\Omega}(x),$ which implies that $x\in \Omega$ and $r\ge f(x)\ge f(\bar x).$ Then, for each $ \epsilon\ge 0$,
$$
f(\bar x)-r\le \epsilon\Vert(x,r)-(\bar x, f(\bar x))\Vert\; \text{ \rm for all } (x, r)\in \epi f_{\Omega}\cap\B((\bar x, f(\bar x)),\delta),
$$
so
$$
\limsup_{(x,r)\xrightarrow{{\rm epi}f_\Omega}(\bar x,f(\bar x))}\dfrac{\langle (0,-1), (x,r)-(\bar x, f(\bar x))\rangle}{\Vert (x,r)-(\bar x,f(\bar x))\Vert}\le \epsilon,
$$
which means that $(0,-1)\in \widehat N_{\epsilon}((\bar x,f_\Omega(\bar x)),\epi f_\Omega)$ or equivalently, $ 0\in \widehat{\partial}_{\epsilon}f_{\Omega}(\bar x).$ Moreover, we see that  $0\in  T(\bar x,\Omega)$, so $0\in J(0)\subset J(T(\bar x,\Omega)).$ Consequently, $0\in \widehat{\partial}^{\epsilon}_{\Omega}f(\bar x)$, i.e.,  \eqref{prob1-eq1} holds. The relation in \eqref{prob1-eq1-limit} is directly followed by combining \eqref{prob1-eq1} with $\epsilon:=0$ and the inclusion in \eqref{sub-relation}.
\end{proof}

As demonstrated  in the  proposition above, using the relative subdifferentials, we establish necessary optimality conditions for  problem~\eqref{prob1} {\it  without} any regularities or qualification conditions, except for  the standard assumptions on the relatively lsc property of  $f$ and the closedness of $\Omega$. The following example illustrates that the Fermat-type optimality rules in  Proposition~\ref{thm2}  not only work well for a {\it non-subdifferentiable} (in the sense of \cite{Mor06,Roc98}) setting but also perform better than some existing ones for the {\it subdifferentiable} framework.

\begin{exam} (i)  [{\it Non-subdifferentiablity setting}] Let $\Omega:=[0,\infty)$, and	consider the problem
	\begin{align}\label{E1} \min \{ f(x)\mid x\in \Omega\},\tag{E1}\end{align} where $f:\R\to\bar\R$ is  given by
	$$f(x):=\begin{cases}
		-\infty&\text{ \rm if } x<0;\\
		0&\text{ \rm if } x=0;\\
		\infty&\text{ \rm if } x>0.
	\end{cases}$$ 	It is clear that $\bar x:=0$ is a  local optimal solution to problem~\eqref{E1}. By direct calculation, we find $f_{\Omega}=\begin{cases}
	0&\text{ \rm if } x=0;\\
	\infty&\text{ \rm otherwise }
\end{cases}$ and $\partial_{\Omega}f(\bar x)=\widehat{\partial}^{\epsilon}_{\Omega}f(\bar x)=[0,\infty)$  for all $ \epsilon\ge 0$. This means that the Fermat-type rules in \eqref{prob1-eq1} and \eqref{prob1-eq1-limit} hold for   problem~\eqref{E1}. However,
  $f$ is not subdifferentiable in the sense of \cite{Mor06,Roc98},   the optimality conditions in \cite[Theorem~8.15]{Roc98}, and \cite[Proposition~5.3]{Mor06} are {\it not applicable} to  this problem. \\
(ii) [{\it Subdifferentiablity setting}] Let $\Omega:=(-\infty,0]$, and consider the problem
	\begin{align}\label{E2} \min\{ f(x)\mid x\in \Omega\},\tag{E2}\end{align} where $f:\R\to\R$ is defined by
$f(x):=-\vert x\vert $ for $x\in \R.$

Considering $\bar x:=0,$ we see that $\partial_{\Omega}f(\bar x)=\{1\}$, so $0\notin \partial_{\Omega}f(\bar x).$ Invoking now the Fermat-type optimality rule of \eqref{prob1-eq1-limit}, we assert that
$\bar x$ is {\it not} a local optimal solution to   problem~\eqref{E2}. However, some existing optimality conditions such as \cite[Theorem~8.15]{Roc98} and \cite[Proposition~5.3]{Mor06} are {\it not able} to recognize the {\it non-optimality} of $\bar x$  due to the fulfilment of their necessary condition: $0\in \partial f(\bar x)+N(\bar x,\Omega),$ where $\partial $ and $N$ stand for the limiting subdifferential and limiting normal cone, respectively, in the sense of  \cite{Mor06,Roc98}.
\end{exam}

We now demonstrate that the locally Lipschitz continuity relative to a set of a function  around a reference point is {\it sufficient } for the {\it non-emptiness} of the relative  limiting subdifferential at that reference point.

\begin{thm}\label{theo-non} Let  $f:X\to \bar{\R}$ be  lsc relative to a set $\Omega\subset X$ and $\Omega $ be locally closed around $\bar x\in \Omega\cap \dom f.$ If $f$ is locally Lipschitz relative to $\Omega$ around $\bar x$, then $\partial_{\Omega} f(\bar x)\ne \emptyset.$
\end{thm}

To prove this theorem, we need to establish some relationships between the locally Lipschitz property  relative to a set of  a function and the relative subdifferentials as in the following lemma.

\begin{lem}\label{lem31}
Let  $f:X\to \bar{\R}$ be  lsc relative to a set $\Omega\subset X$, and let $\Omega $ be locally closed around $\bar x\in \Omega\cap \dom f.$ Then the following statements hold:
	
	{\rm (i)} $f$ is locally Lipschitz relative to $\Omega$ around $\bar x$ with a modulus $\ell\ge 0$ if and only if $\mathcal E^{f}:\mathbb R^n\rightrightarrows \R$ satisfies the Aubin property relative to $\Omega$ around $(\bar x,f(\bar x))$ with the modulus $\ell,$ where $\mathcal E^f$ is a {\it profile mapping}  of $f$  defined by
	$$
	\mathcal E^f(x):=\begin{cases}
		\left\{\alpha\in \R\mid \alpha\ge f(x)\right\}& \text{ \rm if } x\in \dom f,\\
		\emptyset & \text{ \rm otherwise.}
	\end{cases}
	$$
	
	{\rm (ii)} The locally Lipschitz property relative to $\Omega$ of $f$ around $\bar x$  with a modulus $\ell\ge 0$ guarantees the existence of $\eta>0$ satisfying
	\begin{align}\label{lem31-ii-eq1}
		\Vert x^*\Vert\le \ell| y^*|+\epsilon(1+\ell)\end{align}  for all $(x^*,y^*)\in \widehat{N}_{\epsilon}((x,f(x)),\epi f_{\Omega})\cap [J(T(x,\Omega))\times\R]$, $x\in \B(\bar x,\eta)\cap \Omega$  and $\epsilon\ge 0,$ and particularly,
	\begin{align}\Vert x^*\Vert\le \ell+\epsilon(1+\ell)\;  \label{lem31-ii-eq2}
	\end{align} for all $ x^*\in \widehat{\partial}^{\epsilon}_{\Omega}f(x) $, $ x\in \B(\bar x,\eta)\cap \Omega$ and  $\epsilon\ge 0.$
\end{lem}

\begin{proof}
	{\rm (i)}  Let $f$ be locally Lipschitz relative to $\Omega$ around $\bar x$ with a modulus $\ell\ge 0$. Then there exists a neighborhood $U$ of $\bar x$ such that $\vert f(x)-f(u)\vert\le \ell \Vert x-u\Vert$ for all $x,u\in U\cap \Omega,$ which is equivalent to
$f(x)\in f(u) +\ell \Vert x-u\Vert\B$ for all $x,u\in U\cap \Omega.$ This implies that $f(x)+r\in f(u)+r+\ell \Vert x-u\Vert\B$  for all $x,u\in U\cap \Omega$ and  for all  $r\ge 0,$ which in turn yields that $\mathcal E^{f}(x)\subset \mathcal E^f(u)+\ell\Vert x-u\Vert\B$ for all $x,u\in U\cap\Omega.$ Therefore, for an arbitrary neighborhood $V$ of $f(\bar x),$ we have $\mathcal E^{f}(x)\cap V\subset \mathcal E^f(u)+\ell\Vert x-u\Vert\B$ for all $x,u\in U\cap \Omega,$ which means that $\mathcal E^f$ satisfies the Aubin property relative $\Omega$ around $(\bar x,f(\bar x))$ with   modulus~$\ell.$
	
Now, let $\mathcal E^f$ satisfy the Aubin property relative to $\Omega$ around $(\bar x,f(\bar x))$ with the modulus $\ell$, which  means that   there exist neighborhoods $V$ of $f(\bar x)$ and $U$ of $\bar x$ such that
\begin{equation}\label{lem31-i-eq1}
\mathcal E^f(u)\cap V\subset \mathcal E^f(x) +\ell\Vert u-x\Vert \mathbb B\quad  \forall x,u\in \Omega\cap U.
\end{equation}
If $f$ is not locally Lipschitz relative to $\Omega$ around $\bar x$ with a modulus $\ell\ge 0$, then there exist sequences $x_k$ and $y_k$ in $\Omega$ such that   $x_k,y_k\to\bar x$ and
\begin{align}\label{lem31-i-eq1-a}
\vert f(x_k) - f(y_k)\vert>\ell\Vert x_k-y_k\Vert,\ \forall k\in \N.
\end{align}
Without loss of generality, we can assume that $x_k,y_k\in U$ for all $k\in \N.$ One proves that $f(x_k)\to f(\bar x)$ and $f(y_k)\to f(\bar x)$ as $k\to\infty$. Indeed, by \eqref{lem31-i-eq1}, one has $f(\bar x)\in \mathcal E^f(x_k)+\ell\Vert x_k-\bar x\Vert\B.$ It follows that $f(\bar x)\ge f(x_k)-\ell\Vert x_k-\bar x\Vert$, which further implies that
$$
\limsup_{k\to\infty}f(x_k)\le \limsup_{k\to\infty}(f(\bar x)+\ell\Vert x_k-\bar x\Vert)=f(\bar x).
$$
This together with the lsc property relative to $\Omega$ of $f$ around $\bar x$ entails that $f(x_k)\to f(\bar x)$ as $k\to\infty.$ Similarly, we also have $f(y_k)\to f(\bar x)$ as $k\to\infty$. For each $k\in \N,$ we obtain by \eqref{lem31-i-eq1-a} that $f(x_k)>f(y_k)+\ell\Vert x_k-y_k\Vert$ or $f(y_k)>f(x_k)+\ell\Vert x_k-y_k\Vert$. By the symmetry of $x_k$ and $y_k$, we may assume  $f(x_k)>f(y_k)+\ell\Vert x_k-y_k\Vert$ and  set $\epsilon_k:=f(x_k)-f(y_k)-\ell\Vert x_k-y_k\Vert>0$. Thus
\begin{equation}\label{lem31-i-eq1a}
f(y_k)+\frac{\epsilon_k}{2}<f(x_k)-\ell\Vert x_k-y_k\Vert.
\end{equation}
Since $f(y_k)\to f(\bar x)$ and $\epsilon_k\to 0$  as $k\to\infty$, it holds that $f(y_k)+\frac{\epsilon_k}{2}\in V$ for sufficiently large $k.$ Now, for a sufficient  large $k,$ \eqref{lem31-i-eq1a} yields that $$\mathcal E^f(y_k)\cap V\nsubseteq \mathcal E^f(x_k)+\ell\Vert x_k-y_k\Vert\B,$$ which contradicts \eqref{lem31-i-eq1}. Thus $f$ is locally Lipschitz relative to $\Omega$ around $\bar x$ with the modulus $\ell$.
	
	{\rm (ii)} Let  $f$ be locally Lipschitz relative to $\Omega$ of $f$ around  $\bar x$ with a modulus $\ell\ge 0.$ In view of  (i),  one sees that   $\mathcal E^f$ satisfies  the Aubin property relative to $\Omega$ around $(\bar x,f(\bar x))$ with   $\ell.$ Invoking  \cite[Lemma~3.2]{MWY2023b}, we find $\delta>0$ such that \begin{equation}
		\label{thm32-eq1}
		\Vert x^*\Vert\le \ell| y^*|+\epsilon(1+\ell) \; \text{ \rm for all } (x^*,y^*)\in \widehat{N}_{\epsilon}((x,y),\gph\mathcal E^f\cap (\Omega\times\R))\cap \big(J(T(x,\Omega))\times\R\big),
	\end{equation}
	provided that $\epsilon\ge 0$ and $(x,y)\in [\B(\bar x,\delta)\times\B(f(\bar x),\delta)]\cap\gph \mathcal E^f\cap(\Omega\times\R).$ By the locally Lipschitz continuity relative to $\Omega$ of $f$ around $\bar x$, one asserts that   there exists $r>0$ such that
$$
\vert f(x)-f(\bar x)\vert\le \ell\Vert x-\bar x\Vert,\; \forall x\in \Omega\cap\B(\bar x,r).
$$
Setting $\eta:=\min\{r,\frac{\delta}{\ell+1}\},$ we have $\vert f(x)-f(\bar x)\vert<\delta$ for all $x\in \Omega\cap\B(\bar x,\eta),$ which demonstrates that  $f(x)\in \B(f(\bar x),\delta)$ for all $x\in \Omega\cap\B(\bar x,\eta).$ Thus \eqref{thm32-eq1} ensures that
\begin{equation*}
\Vert x^*\Vert\le \ell| y^*|+\epsilon(1+\ell) \; \text{ \rm for all } (x^*,y^*)\in \widehat{N}_{\epsilon}((x,f(x)),\gph\mathcal E^f\cap (\Omega\times\R))\cap \big(J(T(x,\Omega))\times\R\big).
\end{equation*}
for all $x\in \Omega\cap\B(\bar x,\eta)$. Since $\gph\mathcal E^f\cap (\Omega\times\R)=\epi f_{\Omega},$ one sees that inequality \eqref{lem31-ii-eq1} holds for all $(x^*,y^*)\in \widehat{N}_{\epsilon}((x,f(x)),\epi f_{\Omega})\cap \big(J(T(x,\Omega))\times\R\big)$  and $x\in\Omega\cap  \B(\bar x,\eta).$
	
	When $ x^*\in \widehat{\partial}^{\epsilon}_{\Omega}f(x) $ with $ x\in \B(\bar x,\eta)\cap \Omega$ and  $\epsilon\ge 0,$ it holds that $(x^*,-1)\in \widehat{N}_{\epsilon}((x,f(x)),\epi f_{\Omega})$ and $x^*\in J(T(x,\Omega))$. We use \eqref{lem31-ii-eq1} with $y^*:=-1$ to reach inequality \eqref{lem31-ii-eq2}, that is, $\Vert x^*\Vert\le \ell+\epsilon(1+\ell),$ whenever $(x^*,-1)\in \widehat{N}_{\epsilon}((x,f(x)),\epi f_{\Omega})\cap \big(J(T(x,\Omega))\times\R\big),$$ x\in \B(\bar x,\eta)\cap \Omega$ and $\epsilon\ge 0$.
\end{proof}

\begin{proof}[Proof of Theorem~\ref{theo-non}]
	Let $f$ be locally Lipschitz relative to $\Omega$ around $\bar x$ with a modulus $\ell\ge 0$, and  let ${\rm lip}_{\Omega}f(\bar x)$ be defined as in \eqref{lips-con}.	 Using Lemma~\ref{lem31}~(ii) with $\epsilon=0$, we find $\eta>0$ such that
	\begin{equation}\label{subdifferentiability-eq1}
		\Vert x^*\Vert\le \ell\vert y^*\vert
	\end{equation} for all  $(x^*,y^*)\in \widehat{N}((x,f(x)),\epi f_{\Omega})\cap [J(T(x,\Omega))\times\R]$ and $x\in \B(\bar x,\eta)\cap \Omega.$

Clearly,  ${\rm lip}_{\Omega}f(\bar x)\ge~0$. If ${\rm lip}_{\Omega}f(\bar x)=0$, then, for each $k\in\N$, $f$ is locally Lipschitz relative to $\Omega$ around $\bar x$ for  modulus $\epsilon_k:=\frac{1}{k}$, where   $k\in \N$. Then, for each $k\in \N,$ there exists $\delta_k>0$ such that
$$
\vert f(x)-f(u)\vert\le \epsilon_{k}\Vert x-u\Vert,\; \forall x,u\in \Omega\cap\B(\bar x,\delta_k),
$$
which implies, in particular, that $f(\bar x)-f(x)\le \epsilon_{k}\Vert x-\bar x\Vert$ for any $x\in\Omega\cap\B(\bar x,\delta_k)$.  Note that, for each $(x,r)\in {\rm epi}f_{\Omega},$ it follows that $x\in\Omega$ and $r\ge f(x).$ Hence,
$$
\limsup_{(x,r)\xrightarrow{{\rm epi}f_{\Omega}}(\bar x,f(\bar x))}\dfrac{\langle (0,-1),(x,r)-(\bar x,f(\bar x))\rangle}{\Vert (x,r)-(\bar x,f(\bar x))\vert}\le \limsup_{x\xrightarrow{\Omega\cap\B(\bar x,\delta_k)}\bar x}\dfrac{f(\bar x)-f(x)}{\Vert x-\bar x\Vert}\le \epsilon_{k}.
$$
Therefore, $(0,-1)\in \widehat{N}_{\epsilon_{k}}((\bar x,f(\bar x)),\epi f_{\Omega})$, or equivalently, $0\in \widehat{\partial}_{\epsilon_{k}}f_\Omega(\bar x).$ Furthermore, we see that  $0\in  J(T(\bar x,\Omega))$, so $0\in \widehat{\partial}^{\epsilon_k}_\Omega f(\bar x).$ This proves that  $0\in \partial_{\Omega}f(\bar x)$ by taking $x_k:=\bar x, k\in\N,$ in the definition of the relative limiting subdifferential by \eqref{prop1-eq2}. Consequently, $\partial_{\Omega}f(\bar x)\ne \emptyset.$

If ${\rm lip}_{\Omega}f(\bar x)>0$, then $f$ is not locally Lipschitz relative to $\Omega$ around $\bar x$ with the modulus $\bar{\ell}:=\frac{{\rm lip}_{\Omega}f(\bar x)}{2}.$  Using  Lemma~\ref{lem31}~(ii) with $\epsilon=0,$ we find sequences  $x_k\in \B(\bar x,\epsilon_{k})\cap\Omega\subset \B(\bar x,\eta)\cap \Omega$ and $(x_k^*,y_k^*)\in \widehat{N}((x_k,f(x_k)),\epi f_{\Omega})\cap \big(J(T(x_k,\Omega))\times\R\big)$ such that
 \begin{equation}\label{subdifferentiability-eq2}
 	\Vert x_k^*\Vert >\bar{\ell}\vert y_k^*\vert,
 \end{equation}
 where $\epsilon_k:=\frac{1}{k}$ for  $k\in \N$.

Letting $k\in \N$, we obtain by $(x_k^*,y_k^*)\in\widehat{N}((x_k,f(x_k)),\epi f_{\Omega})$ that, for any $\alpha>0$, there exists $\delta>0$ such that
$$
\langle x_k^*, x-x_k\rangle+y_k^*(r-f(x_k))\le \alpha\Vert (x-x_k,r-f(x_k)\Vert,\; \forall (x,r)\in \epi f_\Omega\cap[\B(x_k,\delta)\times\B(f(x_k),\delta)].
$$
Considering $x:=x_k$ and $r:=f(x_k)+\frac{\delta_k}{k},$ we have  $y_k^*\le \alpha.$ Letting $\alpha\to 0,$ we obtain that $y_k^*\le 0$. If $y_k^*=0$ for all sufficiently large $k,$ then, by \eqref{subdifferentiability-eq2}, $(x_k^*,0)\in \widehat{N}((x_k,f(x_k)),\epi f_{\Omega})\cap \big(J(T(x_k,\Omega))\times\R\big)$ and $\Vert x_k^*\Vert>0$ for such sufficiently large $k$. This contradicts to the fact that $\Vert x_k^*\Vert\le 0$ for such sufficiently large $k$ by virtue of   \eqref{subdifferentiability-eq1}.

By taking a subsequence if necessary, we may assume that such that $y_{k}^*< 0$ for all $k\in\N.$ From $(x_{k}^*,y_{k}^*)\in \widehat{N}((x_k,f(x_k)),\epi f_{\Omega})\cap \big(J(T(x_k,\Omega))\times\R\big), k\in \N$, one has
$$
(\frac{x_{k}^*}{-y_{k}^*},-1)\in \widehat{N}((x_k,f(x_k)),\epi f_{\Omega})\cap \big(J(T(x_k,\Omega))\times\R\big),
$$
which entails that
\begin{align}\label{them-1}
(\frac{x_{k}^*}{-y_{k}^*},-1)\in \widehat{N}_{\epsilon_k}((x_k,f(x_k)),\epi f_{\Omega})\cap \big(J(T(x_k,\Omega))\times\R\big),
\end{align}
due to  $\widehat{N}((x_k,f(x_k)),\epi f_{\Omega})\subset \widehat{N}_{\epsilon_k}((x_k,f(x_k)),\epi f_{\Omega}). $ Observe  by \eqref{them-1} that $\frac{x_{k}^*}{-y_{k}^*}\in \widehat{\partial}^{\epsilon_k}_{\Omega}f(x_{k})$ for $k\in \N.$ Moreover, by \eqref{subdifferentiability-eq1}, $\Vert \frac{x_{k}^*}{-y^*_k}\Vert\le \ell,$ $k\in \N,$ which yields that the sequence $\{\frac{x_{k}^*}{-y_{k}^*}\}$ is bounded,  so we can assume  without loss of generality that  $\{\frac{x_{k}^*}{-y_{k}^*}\}$ is  weak* convergent to some $ \tilde x^*\in X^*$.  Consequently, $\tilde x^*\in \partial_{\Omega}f(\bar x)$. Thus
  $\partial_{\Omega}f(\bar x)\ne \emptyset,$  which completes the proof of the theorem.
\end{proof}

The following theorem  establishes some calculus rules for the relative subdifferentials.

\begin{thm}\label{cal-1}
Let  $f: X \rightarrow \bar{\mathbb{R}}$ and $\Omega\subset X$ be such that $f$ is lsc relative to $\Omega$ around $\bar{x}\in \Omega\cap\dom f$ and  $\Omega$ is locally closed around $\bar{x}$. For  $\epsilon\ge 0$ and $\lambda> 0$, we have
\begin{equation}\label{cal-1-eq0}
\lambda \widehat{\partial}^{\hat{\epsilon}}_{\Omega}f(\bar x)\subset \widehat{\partial}^{\epsilon}_{\Omega}(\lambda f)(\bar x) \subset \lambda \widehat{\partial}^{\tilde{\epsilon}}_{\Omega}f(\bar x),
\end{equation}
where $\hat{\epsilon}:=\frac{\epsilon}{\max\{\lambda,1\}}$ and $\tilde{\epsilon}:=\epsilon\cdot\max\{\frac{1}{\lambda}, 1\}.$ Consequently, the following calculus rules hold:\\
	{\rm (i)} $\widehat{\partial}_{\Omega}(\lambda f)(\bar x)=\lambda\widehat{\partial}_{\Omega}f(\bar x)$ for $\lambda>0.$\\
	{\rm (ii)} $\partial_{\Omega}(\lambda f)(\bar x)=\lambda \partial_{\Omega}f(\bar x)$ for $\lambda>0.$
\end{thm}

\begin{proof}
Let  $\epsilon\ge 0$ and $\lambda> 0$. We first justify that
\begin{equation}\label{cal-1-eq2}
\widehat{\partial}^{\epsilon}_{\Omega}(\lambda f)(\bar x) \subset \lambda \widehat{\partial}^{\tilde{\epsilon}}_{\Omega}f(\bar x),
\end{equation}
where $\tilde{\epsilon}:=\epsilon\cdot\max\{\frac{1}{\lambda}, 1\}.$ To see this, one takes $x^*\in \widehat{\partial}^{\epsilon}_{\Omega}(\lambda f)(\bar x).$ By definition, one has $x^*\in J(T(\bar x,\Omega))$ and  $x^*\in \widehat{\partial}_{\epsilon}(\lambda f)_{\Omega}(\bar x)$. The last relation yields that $(x^*,-1)\in \widehat{N}_{\epsilon}((\bar x,f(\bar x)),\epi (\lambda f)_{\Omega})$, so
\begin{equation}\label{cal-1-eq1}
\limsup_{(u,r)\xrightarrow{{\rm epi}\, (\lambda f)\cap(\Omega\times\R)}(\bar x,\lambda f(\bar x))}\dfrac{\langle (x^*,-1), (u,r)-(\bar x,\lambda f(\bar x))\rangle}{\Vert (u,r)-(\bar x,\lambda f(\bar x))\Vert}\le \epsilon.
\end{equation}
We first prove that, for any $(u,r)\in {\rm epi}\, (\lambda f)\cap(\Omega\times\R)$,
\begin{equation}\label{cal-1-eq001}
\max\{1,\lambda\}\Vert (u,\frac{r}{\lambda})-(\bar x,f(\bar x))\Vert\ge \Vert (u,r)-(\bar x,\lambda f(\bar x))\Vert.
\end{equation}
Indeed, if $\lambda\in (0,1]$, then one sees that
\begin{align*}
\left(u-\bar x,r-\lambda f(\bar x)\right)
&=\lambda\left(u-\bar x,\frac{r}{\lambda}-f(\bar x)\right)+(1-\lambda)(u-\bar x,0),
\end{align*}
which entails that
\begin{align*}
\lambda\left\Vert\left( u-\bar x,\frac{r}{\lambda}- f(\bar x)\right)\right\Vert&\ge \left\Vert( u-\bar x,r-\lambda f(\bar x))\right\Vert-(1-\lambda)\Vert (u-\bar x,0)\Vert\\
&\ge \left\Vert( u-\bar x,r-\lambda f(\bar x))\right\Vert-(1-\lambda)\Vert (u-\bar x,r-\lambda f(\bar x))\Vert\\
&=\lambda\Vert (u-\bar x,r-\lambda f(\bar x))\Vert.
\end{align*}
Thus
\begin{equation}
\label{cal-1-eq0002}
\left\Vert\left( u-\bar x,\frac{r}{\lambda}- f(\bar x)\right)\right\Vert\ge \Vert (u-\bar x,r-\lambda f(\bar x))\Vert.
\end{equation}
Similarly, if $\lambda>1,$ then
\begin{align*}
\left(u-\bar x,r-\lambda f(\bar x)\right)=\left(u-\bar x,\frac{r}{\lambda}-f(\bar x)\right)+(\lambda-1)(0,\frac{r}{\lambda}-f(\bar x)),
\end{align*}
which implies that
\begin{align*}
\left\Vert \left(u-\bar x,\frac{r}{\lambda}- f(\bar x)\right)\right\Vert&\ge\Vert (u-\bar x,r-\lambda f(\bar x))\Vert-\frac{\lambda-1}{\lambda}\Vert (0,r-\lambda f(\bar x))\Vert\\
&\ge  \frac{1}{\lambda}\Vert (u-\bar x,r-\lambda f(\bar x))\Vert.
\end{align*}
This ensures that
\begin{equation}\label{cal-1-eq0001}
\lambda \left\Vert \left(u-\bar x,\frac{r}{\lambda}- f(\bar x)\right)\right\Vert\ge \Vert (u-\bar x,r-\lambda f(\bar x))\Vert.
\end{equation}
Combining  \eqref{cal-1-eq0002} and \eqref{cal-1-eq0001}  justifies  \eqref{cal-1-eq001}.

We now justify that	
\begin{equation}\label{cal-1-eq002}
\limsup_{(u,r)\xrightarrow{{\rm epi}\, (\lambda f)\cap(\Omega\times\R)}(\bar x,\lambda f(\bar x))}\dfrac{\langle (\frac{x^*}{\lambda},-1), (u,\frac{r}{\lambda})-(\bar x,f(\bar x))\rangle}{\Vert (u,\frac{r}{\lambda})-(\bar x,f(\bar x))\Vert}\le \tilde{\epsilon}.
\end{equation}
To do this, let us consider any $(u,r)\in {\rm epi}\, (\lambda f)\cap(\Omega\times\R)$ with $(u,\frac{r}{\lambda})\neq (\bar x,f(\bar x))$. If
$$
\langle (\frac{x^*}{\lambda},-1), (u,\frac{r}{\lambda})-(\bar x,f(\bar x))\rangle<0,
$$
then
$$
\dfrac{\langle (\frac{x^*}{\lambda},-1), (u,\frac{r}{\lambda})-(\bar x,f(\bar x))\rangle}{\Vert (u,\frac{r}{\lambda})-(\bar x,f(\bar x))\Vert}< 0\le \tilde\epsilon.
$$
Otherwise, $\langle (\frac{x^*}{\lambda},-1), (u,\frac{r}{\lambda})-(\bar x,f(\bar x))\rangle\ge 0$,  we obtain by  \eqref{cal-1-eq001} that
$$
\dfrac{\langle (\frac{x^*}{\lambda},-1), (u,\frac{r}{\lambda})-(\bar x,f(\bar x))\rangle}{\Vert (u,\frac{r}{\lambda})-(\bar x,f(\bar x))\Vert}\le \max\{1,\lambda\}\dfrac{\frac{1}{\lambda}\langle (x^*,-1), (u,r)-(\bar x,\lambda f(\bar x))\rangle}{\Vert (u,r)-(\bar x,\lambda f(\bar x))\Vert}.
$$
Taking  now \eqref{cal-1-eq1} into account, we obtain
$$
\limsup_{(u,r)\xrightarrow{{\rm epi}\, (\lambda f)\cap(\Omega\times\R)}(\bar x,\lambda f(\bar x))}\dfrac{\langle (\frac{x^*}{\lambda},-1), (u,\frac{r}{\lambda})-(\bar x,f(\bar x))\rangle}{\Vert (u,\frac{r}{\lambda})-(\bar x,f(\bar x))\Vert}\le \dfrac{\max\{1,\lambda\}}{\lambda}\cdot\epsilon=\max\{\frac{1}{\lambda},1\}\cdot\epsilon=\tilde\epsilon.
$$
Consequently,  \eqref{cal-1-eq002} holds. This fact ensures that
\begin{equation*}%
\limsup_{(u,\alpha)\xrightarrow{{\rm epi}\,f_\Omega}(\bar x,f(\bar x))}\dfrac{\langle (\frac{x^*}{\lambda},-1), (u,\alpha)-(\bar x,f(\bar x))\rangle}{\Vert (u,\alpha)-(\bar x,f(\bar x))\Vert}\le \tilde{\epsilon},
\end{equation*}
so $(\frac{x^*}{\lambda},-1)\in \widehat{N}_{\tilde{\epsilon}}((\bar x,f(\bar x)),\epi f_{\Omega}).$ This means that  $\frac{x^*}{\lambda}\in \widehat{\partial}_{\tilde{\epsilon}}f_{\Omega}(\bar x).$ In view of  $x^*\in J(T(\bar x,\Omega))$ and the homogeneous property of $J$,  we conclude that $$\frac{x^*}{\lambda}\in \widehat{\partial}_{\tilde{\epsilon}}f_{\Omega}(\bar x)\cap J(T(\bar x,\Omega)),$$ which yields that    $x^*\in \lambda \widehat{\partial}^{\tilde{\epsilon}}_{\Omega}f(\bar x)$, so \eqref{cal-1-eq2} is valid.

To prove \eqref{cal-1-eq0}, it remains to prove that \begin{equation}
	\label{cal-1-eq4}
	\lambda \widehat{\partial}^{\hat{\epsilon}}_{\Omega}f(\bar x)\subset \widehat{\partial}^{\epsilon}_{\Omega}(\lambda f)(\bar x),
\end{equation}  	
where	$\hat{\epsilon}:=\frac{\epsilon}{\max\{\lambda,1\}}$. To see this, applying   inclusion \eqref{cal-1-eq2} for $\hat\epsilon$, $\hat f:=\lambda f$, and  $\hat \lambda:=\frac{1}{\lambda}$, we obtain
$$
\widehat{\partial}^{\hat{\epsilon}}_{\Omega}f(\bar x)=\widehat{\partial}^{\hat{\epsilon}}_{\Omega}(\hat\lambda \hat f)(\bar x) \subset \hat\lambda \widehat{\partial}^{\bar{\epsilon}}_{\Omega}\hat f(\bar x)=\frac{1}{\lambda}\widehat{\partial}^{\bar{\epsilon}}_{\Omega}(\lambda f)(\bar x),
$$
where $\bar{\epsilon}:=\hat{\epsilon}\cdot\max\{1,\frac{1}{\hat\lambda}\}=\epsilon.$ This indicates  that \eqref{cal-1-eq4} holds,   so we obtain \eqref{cal-1-eq0}.
	
Clearly,	the assertion~(i) is valid by invoking \eqref{cal-1-eq0} with $\epsilon=0.$ To prove (ii), we let $x^*\in \partial_{\Omega}(\lambda f)(\bar x).$ It follows that  there exist sequences $\epsilon_k\downarrow 0$, $x_k\xrightarrow[f]{\Omega}\bar x$, and $x_k^*\rightharpoonup^* x^*$ such that $x_k^*\in \widehat{\partial}^{\epsilon_k}_{\Omega}(\lambda f)(x_k)$ for all $k\in \N$.

By $x_k\xrightarrow[f]{\Omega}\bar x$, the current assumptions are satisfied at $x_k$ near $\bar x$ and we may assume without loss of generality  that these assumptions hold for all $k\in\N.$ By \eqref{cal-1-eq0}, one has $x_k^*\in \lambda \widehat{\partial}^{\tilde{\epsilon}_k}_{\Omega}f(x_k)$ with $\tilde{\epsilon}_k:=\epsilon_k\cdot\max\{1,1/\lambda\}.$ Then, we find  $\tilde{x_k}^*\in \widehat{\partial}^{\tilde{\epsilon}_k}_{\Omega}f(x_k)$ such that $x_k^*=\lambda \tilde x_k^*$ for all $k\in\N.$
Since $x_k^*\rightharpoonup^*x^*,$   sequence $\{x_k^*\}$  is bounded, so is   $\{\tilde x_k^*\}$. We can assume that $\tilde x_k^*\rightharpoonup^* \tilde x^*$ for some $\tilde x^*\in X^*.$ Then,  it holds $\tilde x^*\in \partial_{\Omega}f(\bar x)$, which proves that   $x^*=\lambda \tilde x^*\in \lambda\partial_{\Omega}f(\bar x).$ Thus
	\begin{equation}
		\label{cal-2-eq1}
		\partial_{\Omega}(\lambda f)(\bar x)\subset \lambda\partial_{\Omega}f(\bar x).
	\end{equation}

To prove the inverse inclusion, we take any $x^*_0\in \lambda\partial_{\Omega}f(\bar x)$, that is,   $x^*_0=\lambda x^*$ for some $x^*\in \partial_{\Omega}f(\bar x)$. Then, there exist sequences $\epsilon_k\downarrow 0$, $x_k\xrightarrow[f]{\Omega}\bar x$,  and $x_k^*\rightharpoonup^* x^*$ such that $x_k^*\in \widehat{\partial}^{\epsilon_k}_{\Omega}f(x_k)$ for all $k\in \N$. Hence, $\lambda x_k^*\in \lambda\widehat{\partial}^{\epsilon_k}_{\Omega}f(x_k)$ for all $k\in \N$.  Similarly, by \eqref{cal-1-eq0}, one has $\lambda x_k^*\in  \widehat{\partial}^{\hat{\epsilon}_k}_{\Omega}(\lambda f)(x_k)$ with $\hat{\epsilon}_k:=\epsilon_k\cdot\max\{1,\lambda\}$ for $k\in \N.$ Therefore, we find $\tilde x_k^*\in \widehat{\partial}^{\hat{\epsilon}_k}_{\Omega}(\lambda f)(x_k)$ such that $\lambda x_k^*=\tilde x_k^*$ for $k\in \N.$ Since $x_k^*\rightharpoonup^* x^*$,  $\{x_k^*\}$ is bounded, then $\{\tilde x_k^*\}$ is bounded as well. We can assume that $\tilde x_k^*\rightharpoonup^* \tilde x^*$ for some $\tilde x^*\in X^*$. Thus  $\tilde x^*\in \partial_{\Omega}(\lambda f)(\bar x)$. Note that  $x^*_0=\lambda x^*=\tilde x^*\in  \partial_{\Omega}(\lambda f)(\bar x)$. One easily see that  the inverse inclusion in
	 \eqref{cal-2-eq1} holds, i.e.,   assertion  (ii) is valid.
\end{proof}

\section{Optimality Conditions and Calculus Rules for Relative Subdifferentials}

In this section, we first employ the proposed relative subdifferentials to establish necessary optimality conditions for   non-subdifferentiable sum optimization problem~\eqref{SP}. Based the optimality conditions obtained, we then derive  exact and approximate  sum rules for computing  the relative subdifferentials.

Let us state {\it approximate} and {\it exact} necessary optimality conditions in terms of the relative subdifferentials for  problem~\eqref{SP}.

\begin{thm}\label{thm33}
Let $\bar x\in \mathcal C$ be a local optimal solution of  problem \eqref{SP}. Let   $f_1: X\to\bar{\R}$ is locally Lipschitz relative to $\mathcal C$ around $\bar x$ with a modulus $\ell\ge 0$, and let $f_2: X\to \bar{\R}$ be lsc relative to $\mathcal C$ around $\bar x$. Then, for any small $\eta> 0$, there exist  $x_i\in \B(\bar x,\eta)\cap\mathcal C$ with $\vert f_i(x_i)-f_i(\bar x)\vert\le \eta, i=1,2$ and $\tilde{\eta}\in (0,4\eta(\ell+1))$ such that
	\begin{equation}\label{thm33-eq1}
		0\in \widehat{\partial}^{\tilde{\eta}}_{\mathcal C}f_1(x_1)+\widehat{\partial}^{\tilde{\eta}}_{\mathcal C}f_2(x_2).
	\end{equation}
Consequently,
	\begin{equation}\label{Fermat-rule-eq}
		0\in \partial_{\mathcal C}f_1(\bar x)+\partial_{\mathcal C}f_2(\bar x).
	\end{equation}
\end{thm}

\begin{proof} Since $\bar x$ is a local optimal solution of problem~\eqref{SP}, one sees that there exists $\delta>0$ such that
\begin{equation}\label{Fermat-rule-eq1}
f_1(\bar x)+f_2(\bar x)\le f_1(x)+f_2(x)\; \forall x\in \B(\bar x,\delta)\cap\mathcal C.
\end{equation}
Let $\tilde f_1(x):=f_1(x)-f_1(\bar x),$ $\tilde f_2(x):=f_2(x)-f_2(\bar x)$ for $x\in X$, and set
$$\Omega_1:=\epi \tilde f_1\cap (\mathcal C\times\R) \; \text{ \rm and } \Omega_2:=\left\{(x,y)\in \mathcal C\times\R\mid \tilde f_2(x)\le -y\right\}.$$
It is easy to see that $(\bar x,0)\in \Omega_1\cap \Omega_2$, $\Pi_{X}(\Omega_1\cup\Omega_2)\subset \mathcal C$ and $\Omega_i$, $i=1,2$ are  closed sets.  Taking $b_k=\frac{1}{k}, k\in \N$ and a neighborhood $U:=\B((\bar x,0),\delta)$, we see that \begin{align}\label{1-sua} (\Omega_1+(0,b_k))\cap\Omega_2\cap U=\emptyset, \forall k\in\N.\end{align}
Otherwise, there is $(x,y)\in (\Omega_1+(0,b_k))\cap\Omega_2\cap U$. Then, one has $$y-b_k\ge f_1(x)-f_1(\bar x),\;  y+f_2(x)-f_2(\bar x)\le 0\; \text{ \rm and } \; x\in \mathcal C\cap\B(\bar x,\delta).$$ It follows that
\begin{align*}
f_1(\bar x)+f_2(\bar x)\ge f_1(x)+f_2(x)+b_k>f_1(x)+f_2(x),
\end{align*}
which is a contradiction to \eqref{Fermat-rule-eq1}. Thus \eqref{1-sua} holds, and we conclude that $(\bar x,0)$ is a local extremal point of  $\{\Omega_1,\Omega_2\}$ relative to $\mathcal C$.

Since $f_1$ is locally Lipschitz relative to $\mathcal{C}$ around  $\bar x$ with a modulus $\ell$, one sees that  $\tilde f_1$ is also locally Lipschitz relative to $\mathcal{C}$ around $\bar x$ with the modulus $\ell$.  Using similar  arguments as in the proof of   Lemma~\ref{lem31} with $\tilde f_1$ (cf.~\eqref{thm32-eq1}), we find $\delta_1>0$ such that $\tilde f_1$ is locally Lipschitz relative to $\mathcal C$ around $x$ with the modulus $\ell$, $\tilde f_2$ is lsc relative to $\mathcal C$ around $x$ for any $x\in\B(\bar x,\delta_1)\cap\mathcal C$, and
\begin{align}
	\label{thm33-eq001}
	&\Vert x^*\Vert\le \ell| y^*|+\epsilon(1+\ell) \; \text{ \rm for all } (x^*,y^*)\in \widehat{N}_{\epsilon}((x,y),\Omega_1)\cap \big(J(T(x,\mathcal{C}))\times\R\big),
\end{align}
whenever $\epsilon\ge 0$ and   $(x,y)\in [\B(\bar x,\delta_1)\times\B(0,\delta_1)]\cap \Omega_1.$

Now, for any $\eta$ with $0<\eta<\min\{\delta,\delta_1,\frac{1}{4(\ell+1)}\},$ by invoking \cite[Theorem~6.2]{MWY2023a}, we find   $(x_i,y_i)\in \Omega_i\cap\B((\bar x,0),\eta)$ and $(x_i^*,y_i^*)\in [\widehat{N}((x_i,y_i),\Omega_i)+\eta(\B_{X^*}\times\B)]\cap [J(T(x_i,\mathcal C))\times\R], i=1,2$ such that
\begin{align}\label{thm33-eq3-ex}
(x_1^*,y_1^*)+(x_2^*,y_2^*)=0 \; \text{ \rm and } \;  \Vert (x_1^*,y_1^*)\Vert + \Vert (x_2^*,y_2^*)\Vert =1.
\end{align} This, in particular, implies that $\Vert (x_1^*,y_1^*)\Vert=\Vert (x_2^*,y_2^*)\Vert=\frac{1}{2}$, $\Vert x_1^*\Vert=\Vert x_2^*\Vert\le \frac{1}{2}$  and  $|y^*_1|=|y^*_2|\le \frac{1}{2}.$  Since  $(x_1^*,y_1^*)\in \widehat{N}((x_1,y_1),\Omega_1)+\eta(\B_{X^*}\times\B)$, we conclude by  \eqref{rem21-eq1} that
\begin{align}\label{th-1b}(x_1^*,y_1^*)\in \widehat{N}_{\eta}((x_1,y_1),\Omega_1).\end{align}
If $\vert y_1^*\vert \le \eta$, then, by using \eqref{thm33-eq001} with $\epsilon=\eta,$ we have $
	\Vert x^*_1\Vert\le \eta(1+2\ell) <\frac{1}{2}-\eta,$ which contradicts  the following inequalities:  $\Vert x_1^*\Vert\ge \Vert(x_1^*,y_1^*)\Vert - \vert y_1^*\vert\ge \frac{1}{2}-\eta$.  Consequently, $\vert y_1^*\vert>\eta$, which also shows that $\vert y_2^*\vert> \eta$. Since $(x_i^*,y_i^*)\in \widehat{N}((x_i,y_i),\Omega_i)+\eta(\B_{X^*}\times\B), i=1,2$, one sees that there exist $(u_i,v_i)\in \widehat{N}((x_i,y_i),\Omega_i)$ and $(b_i,c_i)\in \B_{X^*}\times\B$  such that
\begin{align}\label{th-1a}(x_i^*,y_i^*)=(u_i,v_i)+\eta(b_i,c_i), i=1,2.
\end{align}
Note by $(x_i,y_i)\in\Omega_i, i=1,2$ that $x_i\in\mathcal{C}, i=1,2$, $y_1\ge \tilde f_1(x_1)$, and  $-y_2\ge \tilde f_2(x_2)$. We observe that if  $y_1>\tilde f_1(x_1)$, then $v_1=0.$ Indeed, let $y_1>\tilde f_1(x_1)$. We obtain by $(u_1,v_1)\in \widehat{N}((x_1,y_1),\Omega_1)$ that, for any $\alpha>0,$ there exists $r>0$ such that
\begin{equation}\label{thm31-eq006}
\langle (u_1,v_1),(x,y)-(x_1,y_1)\rangle\le \alpha\Vert(x,y)-(x_1,y_1)\Vert
\end{equation}
for all $(x,y)\in \B((x_1,y_1),r)\cap\Omega_1$, where $\Omega_1:=\epi \tilde f_1\cap (\mathcal C\times\R)$. Let $0<\gamma<\min\{r,y_1-\tilde f_1(x_1)\}$.  Picking $x:=x_1, y:=y_1-\gamma$, we see that $(x_1,y) \in \B((x_1,y_1),r)\cap\Omega_1$. From \eqref{thm31-eq006}, one has
$-v_1\le \alpha$. Similarly,  picking $x:=x_1, y:=y_1+\gamma$, one sees that $(x_1,y) \in \B((x_1,y_1),r)\cap\Omega_1$. It follows from \eqref{thm31-eq006} that $v_1\le\alpha.$ Thus  $|v_1|\le\alpha.$ Since $\alpha>0$ was arbitrarily taken, it holds that $v_1=0.$

By \eqref{th-1a} and the above observation,   if  $y_1>\tilde f_1(x_1)$, then  $\vert y_1^*\vert\le \eta$, which contradicts the fact that $\vert y_1^*\vert>\eta.$ Similarly, we can verify by \eqref{th-1a} and   relation  $(u_2,v_2)\in \widehat{N}((x_2,y_2),\Omega_2)$ that if  $-y_2>\tilde f_2(x_2)$, then  $\vert y_2^*\vert\le \eta$, which contradicts the fact that $\vert y_2^*\vert>\eta.$ Thus $y_1=\tilde f_1(x_1)$ and  $-y_2=\tilde f_2(x_2)$. Moreover, we obtain by $(x_i,y_i)\in \B((\bar x,0),\eta), i=1,2$ that  $x_i\in \B(\bar x,\eta), i=1,2$ and $|y_i|\le\eta, i=1,2.$ and Hence, $ \vert f_i(x_i)-f_i(\bar x)\vert\le \eta, i=1,2$.

We now  recall from \eqref{th-1b}  that
$(x_1^*,y_1^*)\in \widehat{N}_{\eta}((x_1,\tilde f_1(x_1)),\Omega_1)$. Then, for any $\alpha>0,$ there exists $\tilde\delta>0$ such that
$$
\langle (x_1^*,y_1^*),(x,y)-(x_1,\tilde f_1(x_1))\rangle\le (\eta+\alpha)\Vert (x,y)-(x_1,\tilde f_1(x_1))\Vert
$$
for all $(x,y)\in \Omega_1\cap\B(x_1,\tilde\delta)\times\B(\tilde f_1(x_1),\tilde\delta).$
Picking $x:=x_1$ and $y:=\tilde f_1(x_1)+\tilde\delta,$ one has $y_1^*\le \eta+\alpha$, which implies that $y_1^*\le \eta$  as $\alpha>0$ was arbitrarily taken. This fact, together with $\vert y_1^*\vert>\eta$ and $|y^*_1|\le \frac{1}{2}$, implies that $-\frac{1}{2}\le y_1^*<-\eta.$ Since $y_1^*=-y_2^*$, we can conclude that $\eta<y_2^*\le \frac{1}{2}.$ Taking \eqref{rem21-eq1} into account, we obtain by
$(x_i^*,y_i^*)\in \widehat{N}((x_i,y_i),\Omega_i)+\eta(\B_{X^*}\times\B), i=1,2$ that
\begin{equation}\label{thm33-eq4}
\left(\frac{x_1^*}{y_2^*},-1\right)\in \widehat{N}((x_1,y_1),\Omega_1)+\frac{\eta}{y_2^*}\B_{X^*}\times\B \subset \widehat{N}_{\tilde{\eta}}((x_1,y_1),\Omega_1)
\end{equation}
and
\begin{equation} \label{thm33-eq5}
\left(\frac{x_2^*}{y_2^*},1\right)\in \widehat{N}((x_2,y_2),\Omega_2)+\frac{\eta}{y_2^*}\B_{X^*}\times\B\subset \widehat{N}_{\tilde{\eta}}((x_2,y_2),\Omega_2),
\end{equation}
where $y_1=\tilde f_1(x_1)$ and  $y_2=-\tilde f_2(x_2)$ and $\tilde{\eta}:=\frac{\eta}{y_2^*}.$ From \eqref{thm33-eq4}, we see that $\frac{x_1^*}{y_2^*}\in \widehat{\partial}_{\tilde{\eta}}\tilde f_{1\mathcal{C}}(x_1)$ and then $\frac{x_1^*}{y_2^*}\in \widehat{\partial}^{\tilde{\eta}}_{\mathcal{C}}\tilde f_{1}(x_1) $  because of $\frac{x_1^*}{y_2^*}\in J(T(x_1,\mathcal C).$ On account of Remark~\ref{rem31}~(ii) and the definition of $\tilde f_1:=f_1-f_1(\bar x),$  we arrive at $\frac{x_1^*}{y_2^*}\in \widehat{\partial}^{\tilde{\eta}}_{\mathcal{C}}f_1(x_1)$. Note that
$(x,y)\in \Omega_2$ if and only if $(x,-y)\in  \epi \tilde f_2\cap (\mathcal{C}\times\R)$
and
$$(x^*,y^*)\in \widehat{N}_{\tilde{\eta}}((x_2,-\tilde f_2(x_2)),\Omega_2) \text{ \rm if and only if } (x^*,-y^*)\in \widehat{N}_{\tilde{\eta}}((x_2,\tilde f(x_2)),\epi \tilde f_2\cap(\mathcal C\times\R)).$$
Therefore,  one can derive similarly from
 \eqref{thm33-eq5} that  $\frac{x_2^*}{y_2^*}\in \widehat{\partial}^{\tilde{\eta}}_{\mathcal{C}} f_2(x_2)$ and so we get by \eqref{thm33-eq3-ex} that
\begin{equation*}\label{thm33-eq004}
    0=\frac{x_1^*+x_2^*}{y_2^*}\in \widehat{\partial}^{\tilde{\eta}}_{\mathcal{C}}f_1(x_1)+\widehat{\partial}^{\tilde{\eta}}_{\mathcal{C}} f_2(x_2).
\end{equation*}
Using \eqref{thm33-eq001} with $\epsilon:=\eta$, one can deduce that
$$
\frac{1}{2}=\Vert (x^*_1,y^*_1)\Vert\le \Vert x_1^*\Vert+\vert y_1^*\vert\le (\vert y_1^*\vert+\eta)(\ell+1),
$$
which implies  that
$$\vert y_1^*\vert\ge \frac{1}{2(1+\ell)}-\eta>\frac{1}{4(\ell+1)}$$ due to $\eta<\frac{1}{4(\ell+1)}$. Therefore, $\tilde\eta:=\frac{\eta}{y_2^*}=\frac{\eta}{\vert y_1^*\vert}< 4\eta(\ell+1)$  and consequently, \eqref{thm33-eq1} has been justified.

To justify inclusion \eqref{Fermat-rule-eq}, we set $\eta_k:=\frac{1}{4k(\ell+1)}$ for $k\in\N$.  For all $k\in\N$ large enough, we use  \eqref{thm33-eq1} to find   $x_{ik}\in \B(\bar x,\eta_k)\cap\mathcal C$, $\tilde\eta_k\in (0,\frac{1}{k})$ and $x_{ik}^*\in \widehat{\partial}_{\mathcal C}^{\tilde\eta_k}f_i(x_{ik})$ such that $\vert f_i(x_{ik})- f_i(\bar x)\vert\le\eta_k$, $i=1,2$ and
\begin{equation}\label{thm33-eq6}
x_{1k}^*+x_{2k}^*=0.
\end{equation}
Since $f_{1}$ is locally Lipschitz   relative to $\mathcal C$ around $\bar x$ with the modulus $\ell,$ it also has this property around $x_{1k}$ for all sufficiently large $k.$ Thus, by using \eqref{lem31-ii-eq2} in Lemma~\ref{lem31}, one has
$$\Vert x_{1k}^*\Vert\le \ell+\tilde\eta_k(1+\ell),$$ which shows that $\{x_{1k}^*\}$ is bounded. By taking a subsequence if necessary, we can assume that $x_{1k}^*\rightharpoonup^*x^*_1$ for some $x^*_1\in X^*$  as $k\to \infty.$ It proves that $x_1^*\in \partial_{\mathcal{C}} f_1(\bar x).$ Furthermore,   by \eqref{thm33-eq6}, $x_{2k}^*=-x_{1k}^*\rightharpoonup^* -x_1^*$ as $k\to \infty$, we have $-x_1^*\in \partial_{\mathcal C}f_2(\bar x)$. Thus,  inclusion \eqref{Fermat-rule-eq} is established, which completes the proof of the theorem.
\end{proof}

The following illustrates the importance of the assumptions in Theorem~\ref{thm33} for obtaining the approximate and exact necessary optimality conditions. More precisely,   condition~\eqref{Fermat-rule-eq} may go away if the lsc property of $f_2$ relative to $\mathcal{C}$ around the reference point is violated.

\begin{exam} [{\it The importance of assumptions}] Let $\mathcal C:=[-\frac{1}{2}, 0]$ and
	consider problem \begin{equation}\label{example33}
		\min \{f_1(x)+f_2(x)\mid  x\in\mathcal C\},\tag{E3}\end{equation} where $f_i:\R\to\bar\R, i=1,2$ are given respectively by
	$$f_1(x):= \begin{cases}\frac{1}{x+1}-1& \text{ \rm if } x>-1;\\
		\infty& \text{ \rm otherwise,}\end{cases}\; f_2(x):=\begin{cases}
		\infty&\text{ \rm if } x<0;\\
		0&\text{ \rm if } x=0;\\
		-\infty&\text{ \rm if } x>0.
	\end{cases}$$
	We can  check that $\bar x:=0$ is a local optimal solution to  problem \eqref{example33}. By direct calculation, we obtain $\partial_{\mathcal C} f_1(0)=\{-1\}$ and $ \partial_{\mathcal C} f_2(0)=(-\infty, 0]$, so $0\notin \partial_{\mathcal C}f_1(0)+\partial_{\mathcal C}f_2(0),$ which shows that the exact necessary optimality condition in \eqref{Fermat-rule-eq} is not valid for this setting. The reason is that $f_2$ is lsc relative to $\mathcal{C}$ at $\bar x$ but {\it not around} $\bar x.$
\end{exam}
Obviously, the above approximate and exact necessary optimality conditions can be viewed as the generalizations of corresponding results for DC programming frameworks such as those in \cite[Proposition~4.6]{HT1999} and \cite[Theorem~6.1]{DNV2010}. The following example demonstrates how   optimality condition~\eqref{Fermat-rule-eq} works efficiently in both the {\it non-optimality} and {\it optimality}  settings.

\begin{exam} (i) [{\it Optimality setting}] Let $\mathcal C:=[-\frac{1}{2}, 0]$ and
consider the problem
\begin{equation}\label{example33a}
\min \{f_1(x)+f_2(x)\mid  x\in\mathcal C\},\tag{E3}\end{equation} where $f_i:\R\to\bar\R, i=1,2$ are given respectively by
$$f_1(x):= \begin{cases}\frac{1}{x+1}-1& \text{ \rm if } x>-1;\\
\infty& \text{ \rm otherwise,}\end{cases}\; f_2(x):=-|x|,\; x\in\R.$$
We can  check that $\bar x:=0$ is a local optimal solution to   problem \eqref{example33}. On the one hand, by direct calculation, we obtain $\partial_{\mathcal C} f_1(0)=\{-1\}$ and $ \partial_{\mathcal C} f_2(0)=\{1\}$, so $0\in \partial_{\mathcal C}f_1(0)+\partial_{\mathcal C}f_2(0).$
On the other hand, since all the assumptions in Theorem~\ref{thm33} are satisfied, we can employ  the exact necessary optimality condition in \eqref{Fermat-rule-eq} to conclude that
$0\in \partial_{\mathcal C}f_1(0)+\partial_{\mathcal C}f_2(0).$
(ii) [{\it Non-optimality setting}] Let $C:=(-\infty, 0]$ and consider the problem \begin{equation}\label{example32}
\min \{f_1(x)+f_2(x)\mid  x\in\mathcal C\},\tag{E4}\end{equation} where $f_i:\R\to\bar\R, i=1,2$ are given respectively by
$f_1(x):=e^x-1$ and $f_2(x):= -\vert x\vert$ for all $x\in \R.$

Considering $\bar x:=0$, we see that  all the  assumptions in Theorem~\ref{thm33} are fulfilled at $\bar x.$ Moreover, we have
$\partial_{\mathcal C} f_1(0)=\{1\} $ and $ \partial_{\mathcal C} f_2(0)=\{1\}$. Thus $0\notin \partial_{\mathcal C} f_1(0)+\partial_{\mathcal C} f_2(0).$ Invoking now
 Theorem~\ref{thm33}, we conclude that $\bar x$ is {\it not} a local optimal solution to   problem~\eqref{example32}.
\end{exam}

With the help of the approximate and exact optimality conditions obtained in Theorem~\ref{thm33}, we now establish {\it approximate} and {\it exact sum rules} for computing  the relative subdifferentials.

\begin{thm}\label{thm-sumrule}
Let $\Omega\subset X$ be closed and convex and $\bar x\in \Omega.$ Let $f_1: X\to\bar{\R}$ be locally Lipschitz relative to $\Omega$ around $\bar x$ with a modulus $\ell\ge 0$, and let $f_2: X\to \bar{\R}$ be lsc relative to $\Omega$ around $\bar x$.  Suppose that $J(T(x,\Omega))$ is convex for all $x$ near $\bar x$. (This assumption is automatically satisfied if $\Omega$ is a convex set in a Hilbert space $X$.) Then, for any  small $\epsilon\ge 0$, $\eta>\epsilon$,  and $x^*\in \widehat{\partial}^{\epsilon}_{\Omega}(f_1+f_2)(\bar x)$, there exist  $\bar\gamma>0$,  $\bar x_i\in \bar x+\eta\B$ with $\vert f_i(\bar x_i)-f_i(\bar x)\vert\le \eta$, $\eta_i>0$ and $\bar x_i^*\in \widehat{\partial}^{\eta_i}_{\Omega}f_i(\bar x_i),$ $i=1,2$ such that
\begin{equation}\label{fuzzy-sumrule}
x^*\in\bar x_1^*+\bar x_2^*+\bar{\gamma}\B_{X^*},
\end{equation} where $\bar\gamma\to 0$ and $\eta_i\to 0$ as $\eta\to 0.$ Consequently,
\begin{equation}\label{sum-rule-eq}
\partial_{\Omega}(f_1+f_2)(\bar x)\subset \partial_{\Omega}f_1(\bar x)+\partial_{\Omega}f_2(\bar x).
\end{equation}
\end{thm}

\begin{proof}
From the local Lipschitz continuity relative to $\Omega$ of $f_1$ around $\bar x$ with a modulus  $\ell$, we find $\delta>0$ such that
\begin{align}\label{thm33-eq2}
\vert f_1(x)-f_1(y)&\le \ell\Vert x-y\Vert \; \text{ \rm for all } x,y\in \Omega\cap\B(\bar x,\delta).
\end{align}
For any $\epsilon\ge 0, \eta>0$ with $\eta<\min\{\frac{\delta}{3},\frac{1}{4(\ell+1)}\},$ we pick $x^*\in \widehat{\partial}^{\epsilon}_{\Omega}(f_1+f_2)(\bar x)=\widehat{\partial}_{\epsilon}(f_1+f_2)_{\Omega}(\bar x)\cap J(T(\bar x,\Omega))$. By \cite[Proposition~2.2]{MWY2023a}, we find  $\delta_\eta>0$ with $\delta_\eta\le \delta$ such that
\begin{equation}\label{thm42-eq002}
x^* \in J(T(x,\Omega))+\eta\B_{X^*}\; \text{ \rm for all } x\in \B(\bar x,\delta_\eta)\cap\Omega.
\end{equation}
Moreover, it implies from \cite[Theorem~1.86 and Proposition 1.84]{Mor06} that $\bar x$ is a local minimizer of the following function:
$$(f_1+f_2)_{\Omega}(x)-\langle x^*, x-\bar x\rangle+\left(\epsilon+\frac{\epsilon(1+\Vert x^*\Vert)}{1-\epsilon}\right)\Vert x-\bar x\Vert,\, x\in X,$$ which follows that the following function
$$\psi(x):=\left(f_1(x)-\langle x^*, x-\bar x\rangle+\epsilon\left(1+\frac{1+\Vert x^*\Vert}{1-\epsilon}\right)\Vert x-\bar x\Vert\right)+f_2(x),\; x\in X,$$ attains a local minimum on $\Omega$ at $\bar x.$ We observe that the function $$\varphi(x):=f_1(x)-\langle x^*, x-\bar x\rangle+\epsilon\left(1+\frac{1+\Vert x^*\Vert}{1-\epsilon}\right)\Vert x-\bar x\Vert,\; x\in X,$$ is locally Lipschitz relative to $\Omega$ around $\bar x$ with the modulus $\tilde{\ell}:=\epsilon+\Vert x^*\Vert+\epsilon\left(1+\frac{1+\Vert x^*\Vert}{1-\epsilon}\right).$

Let $\tilde\delta>0$ be such that  $\tilde\delta\le \min\{\delta_\eta,\frac{\eta}{2}\}$. According to Theorem~\ref{thm33}, we see that there exist $\tilde\eta\in (0,4\tilde\delta(1+\tilde\ell))$, $x_1, x_2\in \B(\bar x,\tilde\delta)$ with $\vert\varphi(x_1)-\varphi(\bar x)\vert\le \tilde\delta$, $\vert f_2(x_2)-f_2(\bar x)\vert\le \tilde\delta$ and $x_1^*\in \widehat{\partial}^{\tilde{\eta}}_{\Omega}\varphi(x_1)$, $x_2^*\in \widehat{\partial}^{\tilde{\eta}}_{\Omega}f_2(x_2)$ such that
\begin{align}
   x_1^*+x_2^*=0. \label{thm42-eq001}.
\end{align}
Since $x_1^*\in \widehat{\partial}^{\tilde{\eta}}_{\Omega}\varphi(x_1)=\widehat{\partial}_{\tilde{\eta}}\varphi_{\Omega}(x_1)\cap J(T(x_1,\Omega)),$ for any $\epsilon>0,$ we see that there exist $\delta_1\in (0,\tilde\delta)$ such that
\begin{align*}
\langle x_1^*, x-x_1\rangle -(\varphi(x)-\varphi(x_1))&\le (\epsilon+\tilde{\eta})(\Vert x-x_1\Vert+\vert \varphi(x)-\varphi(x_1)\vert),\; \forall x\in \Omega\cap\B(x_1,\delta_1),
\end{align*}
which follows that \begin{align}
\langle x_1^*+x^*, x-x_1\rangle -(\varphi_1(x)-\varphi_1(x_1))&\le (\epsilon+\tilde{\eta})\left(1+\ell+\Vert x^*\Vert+\eta\Big(1+\dfrac{1+\Vert x^*\Vert}{1-\eta}\Big)\right)\Vert x-x_1\Vert\nonumber\\
&=(\epsilon+\tilde{\eta})(1+\tilde{\ell})\Vert x-x_1\Vert, \; \forall x\in \Omega\cap\B(x_1,\delta_1),
\label{thm42-eq003}
\end{align}
where $\varphi_{1}(x):=f_1(x)+\eta\left(1+\frac{1+\Vert x^*\Vert}{1-\eta}\right)\Vert x-\bar x\Vert.$  From \eqref{thm42-eq002}, we can find $\tilde x^*\in J(T(x_1,\Omega))$ and $b^*\in \B_{X^*}$ such that $x^*=\tilde x^*+\eta b^*$, which together with \eqref{thm42-eq003} yields that
\begin{align*}
\langle x_1^*+\tilde x^*, x-x_1\rangle -(\varphi_1(x)-\varphi_1(x_1))&= \langle x_1^*+x^*, x-x_1\rangle -(\varphi_1(x)-\varphi_1(x_1))-\eta \langle b^*, x-x_1\rangle\\
&\le (\epsilon+\tilde{\eta}+\eta)(1+\tilde{\ell})\Vert x-x_1\Vert
\end{align*}
for all $x\in \Omega\cap\B(x_1,\delta_1).$ As $\epsilon>0$ was arbitrarily taken, we assert that
\begin{equation*} 
\liminf_{x\rightarrow x_1}\dfrac{(\varphi_{1})_\Omega(x)-(\varphi_{1})_\Omega(x_1)-\langle x_1^*+\tilde x^*, x-x_1\rangle}{\Vert x-x_1\Vert}\ge - \bar{\eta}:=-(\eta+\tilde{\eta})(1+\tilde{\ell}),
\end{equation*}
which demonstrates that $x_1^*+\tilde x^*\in \widehat{\partial}_{\bar \eta}(\varphi_1)_{\Omega}(x_1)$ due to \cite[Definition~1.83]{Mor06} and \cite[Theorem~1.86]{Mor06}. Moreover, by our assumption that $J(T(x,\Omega))$ is convex for all $x$ near $\bar x$, we can assume that $J(T(x_1,\Omega))$ is convex, which implies that $x_1^*+\tilde x^*\in J(T(x_1,\Omega))$ as both $x_1^*,\tilde x^*\in J(T(x_1,\Omega))$ and $J(T(x_1,\Omega))$ is a convex cone. Hence, $x_1^*+\tilde x^*\in \widehat{\partial}^{\bar\eta}_{\Omega}\varphi_1(x_1).$ Obverse that $\varphi_1$ is local Lipschitz relative to $\Omega$ around $x_1$ with the modulus $\ell+\eta\left(1+\frac{1+\Vert x^*\Vert}{1-\eta}\right)$. According to Lemma~\ref{lem31}, we derive that
\begin{equation}\label{thm42-eq01}
\Vert x_1^*+\tilde x^*\Vert\le \bar\ell:=\ell+\eta\left(1+\frac{1+\Vert x^*\Vert}{1-\eta}\right)+\bar\eta\left(1+\ell+\eta\left(+\frac{1+\Vert x^*\Vert}{1-\eta}\right)\right).
\end{equation}

On the other hand, by the fact $x_1^*+\tilde x^*\in \widehat{\partial}^{\bar\eta}_{\Omega}\varphi_1(x_1)=\widehat{\partial}_{\bar\eta}{\varphi_1}_{\Omega}(x_1)\cap J(T(x_1,\Omega))$ and \cite[Proposition 1.84]{Mor06} (with $\epsilon:=\bar\eta$ and $\gamma:=\eta$), we assert that $x_1$ is a
local minimizer of the following function
$$\left(\eta\Big(1+\frac{1+\Vert x^*\Vert}{1-\eta}\Big)\Vert x-\bar x\Vert -\langle x_1^*+x^*,x-x_1\rangle+(\bar{\eta}+\eta)\Vert x-x_1\Vert\right) +f_1(x) +\delta_{\Omega}(x),$$ which follows that $x_1$ is a local minimizer on $\Omega$ of the function
$$\left(\eta\Big(1+\frac{1+\Vert x^*\Vert}{1-\eta}\Big)\Vert x-\bar x\Vert -\langle x_1^*+x^*,x-x_1\rangle+(\bar{\eta}+\eta)\Vert x-x_1\Vert\right) +f_1(x).$$ Define the function
$$\varphi_2(x):=\eta\Big(1+\frac{\eta(1+\Vert x^*\Vert)}{1-\eta}\Big)\Vert x-\bar x\Vert -\langle x_1^*+x^*,x-x_1\rangle+(\bar{\eta}+\eta)\Vert x-x_1\Vert.$$
Note from \eqref{thm33-eq2} that $f_1$ is locally Lipschitz relative to $\Omega$ around $x_1$. Using \cite[Proposition~2.2]{MWY2023a}, one sees that there exists $\delta_2>0$ such that
\begin{equation}\label{thm42-eq02}
x_1^*+x^*\in J(T(x,\Omega))+\eta\B_{X^*} \; \forall x\in \B(x_1,\delta_2)\cap\Omega.\end{equation}
Set $\tilde\delta_2:=\min\{\delta_1,\delta_2\}$. By using Theorem~\ref{thm33}  (with $\eta:=\tilde\delta_2$), we find $\check{\eta}\in (0,4\tilde\delta_2(1+\ell)),$ $\tilde x_1, \tilde x_2\in \B(x_1,\tilde\delta_2)$ with $\vert f_1(\tilde x_1)-f_1(x_1)\vert\le \tilde\delta_2$, $\vert \varphi_2(\tilde x_2)-\varphi_2(x_1)\vert\le \tilde\delta_2$ and $\hat{x}_1^*\in \widehat{\partial}^{\check{\eta}}_{\Omega}f_1(\tilde x_1)$, $\hat{x}_2^*\in \widehat{\partial}^{\check{\eta}}_{\Omega}\varphi_2(\tilde x_2)$ such that
\begin{align}\label{thm33-eq003}
\hat{x}_1^*+\hat{x}_2^*=0.
\end{align}
Since $\hat{x}_2^*\in \widehat{\partial}^{\check{\eta}}_{\Omega}\varphi_2(\tilde x_2)=\widehat{\partial}^{\check{\eta}}{(\varphi_2)}_{\Omega}(\tilde x_2)\cap J(T(\tilde x_2,\Omega)),$ then there exists, for any $\epsilon>0,$ $\delta_3\in (0,\tilde\delta_2)$ such that
\begin{align*}
\langle \hat{x}^*_2, x-\tilde x_2\rangle-(\varphi_2(x)-\varphi_2(\tilde x_2))&\le (\epsilon+\check{\eta})\left(\Vert x-\tilde x_2\Vert+\vert \varphi_2(x)-\varphi_2(\tilde x_2)\vert\right),\; \forall x\in \B(\tilde x_2,\delta_3)\cap \Omega,
\end{align*}
which follows that
\begin{align*}
&\langle \hat{x}^*_2+x_1^*+x^*, x-\tilde x_2\rangle\\
&\le \left(\eta\Big(1+\frac{1+\Vert x^*\Vert}{1-\eta}\Big)+(\bar{\eta}+\eta)+ (\epsilon+\check{\eta})\Big(\eta\big(1+\frac{1+\Vert x^*\Vert}{1-\eta}\big)+\Vert x_1^*+x^*\Vert+\bar{\eta}+\eta\Big)\right)\Vert x-\tilde x_2\Vert\\
&\le \left(\eta\Big(1+\frac{1+\Vert x^*\Vert}{1-\eta}\Big)+(\bar{\eta}+\eta)+ (\epsilon+\check{\eta})\Big(\eta\big(1+\frac{1+\Vert x^*\Vert}{1-\eta}\big)+\bar\ell+\bar{\eta}+\eta\Big)\right)\Vert x-\tilde x_2\Vert
\end{align*}
for all $x\in \B(\tilde x_2,\delta_3)\cap \Omega,$ where the final inequality holds due to \eqref{thm42-eq01}. Setting
$$\gamma:=\eta\Big(1+\frac{1+\Vert x^*\Vert}{1-\eta}\Big)+(\bar{\eta}+\eta)+ \check{\eta}\Big(\eta\big(1+\frac{1+\Vert x^*\Vert}{1-\eta}\big)+\bar\ell+\bar{\eta}+\eta\Big),$$ one has $\hat x_2^*+x_1^*+x^*\in \widehat{N}_{\gamma}(\tilde x_2,\Omega)$. By the convexity of $\Omega,$ we derive \begin{equation}\label{thm42-eq004}
\hat x_2^*+x_1^*+x^*\in\widehat{N}(\tilde x_2,\Omega)+\gamma\B_{X^*}.
\end{equation}
In view of  \eqref{thm42-eq02}, one has $x_1^*+x^*\in J(T(\tilde x_2,\Omega))+\eta\B_{X^*},$ which means that there exists $b_2\in\B_{X^*}$ such that $\hat x_2^*+x_1^*+x^*+\eta b_2\in J(T(\tilde x_2,\Omega)).$ This implies from the definition of $J$ that there exists $v\in T(\tilde x_2,\Omega)$ such that $$\langle\tilde x_2^*+x_1^*+x^*,v\rangle=\Vert \tilde x_2^*+x_1^*+x^*\Vert^2=\Vert v\Vert^2.$$ Moreover, we find sequences $t_k\to 0^+$ and $v_k\to v$ such that $\tilde x_k:=\tilde x_2+t_kv_k\in \Omega,$ which implies that $v_k=\frac{\tilde x_k-\tilde x_2}{t_k}.$ Thus
\begin{equation}\label{thm42-eq005}
\Vert \tilde x_2^*+x_1^*+x^*\Vert^2=\langle \tilde x_2^*+x_1^*+x^*, v\rangle=\lim_{k\to\infty}\langle \tilde x_2^*+x_1^*+x^*, v_k\rangle
\end{equation}

On the other hand, from \eqref{thm42-eq004}, we can find $b_3\in \B_{X^*}$ such that $\hat x_2^*+x_1^*+x^*+\gamma b_3\in \widehat{N}(\tilde x_2,\Omega),$ which follows that
$$
\langle \hat x_2^*+x_1^*+x^*+\gamma b_3, v_k\rangle=\frac{1}{t_k}\langle \hat x_2^*+x_1^*+x^*+\gamma b_3, \tilde x_k-\tilde x_2\rangle\le 0,
$$
due to the convexity of $\Omega.$ This gives us that
$$
\langle  \hat x_2^*+x_1^*+x^*, v_k\rangle\le -\bar\gamma\langle b_3,v_k\rangle\le \gamma \Vert v_k\Vert,
$$
which follows that $\langle  \hat x_2^*+x_1^*+x^*, v\rangle\le \gamma\Vert v\Vert.$  Picking \eqref{thm42-eq005}, \eqref{thm42-eq001}, and \eqref{thm33-eq003} into account and setting $\bar\gamma:=\sqrt{\gamma\Vert v\Vert}$, we obtain
$$
\hat x_2^*+x_1^*+x^*=-\hat x_1^*-x_2^*+x^*\in \bar\gamma\B_{X^*}.
$$
Setting $\eta_1:=\check{\eta},$ $\eta_2:=\tilde{\eta}$, $\bar x_1:=\tilde{x}_1\in \B(x_1,\tilde\delta_2)\subset \B(x_1,\frac{\eta}{2})\subset\B(\bar x,\eta), \bar x_2:=x_2\in\B(\bar x,\eta)$, $\bar x_1^*:=\hat{x}_1^*\in \widehat{\partial}^{\eta_1}_{\Omega}f_1(\bar x_1)$, and $\bar x_2^*\in \widehat{\partial}^{\eta_2}_{\Omega}f_2(\bar x_2)$, we obtain  relation \eqref{fuzzy-sumrule}. Note that
$$0<\eta_1=\check{\eta}<4\tilde\delta_2(1+\ell)<2\eta(1+\ell),$$
$$
0<\eta_2=\tilde\eta<2\eta\left(1+\ell+\Vert x^*\Vert+\eta\left(1+\frac{1+\Vert x^*\Vert}{1-\eta}\right)\right),
$$
and
\begin{align*}
\bar\gamma&:=\eta\Big(1+\frac{1+\Vert x^*\Vert}{1-\eta}\Big)+(\bar{\eta}+\eta)+ \check{\eta}\Big(\eta\big(1+\frac{1+\Vert x^*\Vert}{1-\eta}\big)+\bar\ell+\bar{\eta}+\eta\Big).
\end{align*}
Thus $\bar\gamma\to 0^+$, $\eta_i\to 0^+$ for $i=1,2$ as $\eta\to 0.$

It remains to prove inclusion \eqref{sum-rule-eq}. Taking $x^*\in \partial_{\Omega}(f_1+f_2)(\bar x),$ one sees that there exist by the definition sequences $\epsilon_k\downarrow 0,$ $x_k\xrightarrow[f]{\Omega}\bar x,$  $x_k^* \rightharpoonup x^*$ such that $x_k^*\in \widehat{\partial}^{\epsilon_k}_{\Omega}(f_1+f_2)(x_k)$ for all $k\in \N$. Without loss of generality, we can assume that $f_1$ is locally Lipschitz continuous relative to $\Omega$ around $x_k$ with the same modulus $\ell$ for any $k\in \N.$ Thus there exist $x_{ik}\in \B(x_k,\epsilon_k)$, $x_{ik}^*\in \widehat{\partial}^{\epsilon_{ik}}_{\Omega}f(x_{ik}),$ for $i=1,2$ and $b_k\in\B_{X^*}$ such that
\begin{equation}\label{thm-sumrule-eq4}
x_{k}^* = x_{1k}^*+x_{2k}^*+\bar{\gamma}_kb_k,\; \forall k\in \N,
\end{equation}
where $\bar\gamma_k\to 0^+,$ $\epsilon_{ik}\to 0^+$ as $\epsilon_k\to 0^+.$ We can also assume that $f_1$ is locally Lipschitz continuous relative to $\Omega$ around $x_{1k}$ for all $k\in \N.$  Using \eqref{lem31}, one has $\Vert x_{1k}^*\Vert\le \ell+\epsilon_{1k}(1+\ell),$ which follows that $\Vert x_{1k}^*\Vert$ is bounded. Without loss of generality, we can assume that $\epsilon_{ik}\downarrow 0$ for $i=1,2$, $x_{1k}^*\rightharpoonup^* x_1^*$ and $b_k\rightharpoonup^* b$ for some $x_1^*\in X^*$ and $b\in \B_{X^*}.$ Combining this with $x_k^*\rightharpoonup^* x^*$ and \eqref{thm-sumrule-eq4}, we can assume that $x_{2k}^*\rightharpoonup^* x_2^*$ for some $x_2^*\in X^*.$ Thus it implies from the definition that $x_1^*\in \partial_{\Omega} f_1(\bar x)$ and $x_2^*\in \partial_{\Omega}f_2(\bar x).$ Passing to the limit in \eqref{thm-sumrule-eq4} as $k\to \infty,$ one has
$
x^*=x_1^*+x_2^*\in  \partial_{\Omega} f_1(\bar x)+ \partial_{\Omega}f_2(\bar x),
$
which means that
$$
\partial_{\Omega}(f_1+f_2)(\bar x)\subset \partial_{\Omega} f_1(\bar x)+ \partial_{\Omega}f_2(\bar x).
$$
Hence the proof of theorem is complete.
\end{proof}

\section{Mean Valued Theorems and Subdifferential Monotonicity via Relative Subdifferentials}

This section is devoted to establishing  mean value theorems and subdifferential  monotonicities via the relative subdifferentials. Our results are established under mild assumptions such as the lsc relative to a set and the locally Lipschitz continuity relative to a set,  so they are applicable  for a wider class of functions.

\begin{thm}\label{thm33a}
Let $f: X\to\bar{\R}$ be finite at both points $a\ne b$, and let $f$ be properly lsc relative to $[a,b],$ where $a\in X$ and $b\in X.$ Assume that for any $c\in [a,b)$ at which the function
$$\varphi(x):=f(x)+\dfrac{f(b)-f(a)}{\Vert b-a\Vert}\Vert x-b\Vert$$
attains its minimum on $[a,b]$. Then, for any sequence $\epsilon_k\downarrow 0,$ there are sequences $\tilde\epsilon_k\to 0^+$, $x_k\xrightarrow[f]{[a,b]}c$ and $x_k^*\in \widehat{\partial}_{[a,b]}^{\tilde\epsilon_k}f(x_k)$ satisfying the following conditions:
\begin{align}
\liminf_{k\to\infty}\langle x_k^*, b-a\rangle\ge f(b)-f(a),\label{thm33a-eq00}\\
\label{thm33a-eq0}
\liminf_{k\to\infty}\langle x_k^*, b-x_k\rangle\ge \dfrac{\Vert b-c\Vert}{\Vert b-a\Vert}(f(b)-f(a)).
\end{align}
If, in addition, $c\ne a,$ then
\begin{align}\label{thm33a-eq000}
\lim_{k\to\infty}\langle x_k^*, b-a\rangle = f(b)-f(a).
\end{align}
\end{thm}

\begin{proof}
Since $c$ is a local solution to problem $\min_{x\in[a,b]}\varphi(x)$, for each $\epsilon_k\downarrow 0$, we find, by  Theorem~\ref{thm33}, sequences $\tilde\epsilon_k\in \Big(0,4\epsilon_k(1+\frac{f(b)-f(a)}{\Vert b-a\Vert})\Big),$ $\{x_{ik}\}\subset\B(c, \epsilon_k)\cap [a,b],$ $i=1,2$ and $x_{k}^*\in \widehat{\partial}^{\tilde\epsilon_k}_{[a,b]}f(x_{1k})\cap[-\widehat{\partial}^{\tilde\epsilon_k}_{[a,b]}g(x_{2k})]$ with $g(x):=\frac{f(b)-f(a)}{\Vert b-a\Vert}\Vert x-b\Vert$ for any $x\in X.$ This implies that $-x_k^*\in \widehat{\partial}^{\tilde\epsilon_k}g_{[a,b]}(x_{2k}).$ According to the convexity of $g,$ the definition of geometric $\epsilon$-subdifferentials, and \cite[Proposition~1.3]{Mor06}, one has
\begin{equation}\label{thm33a-eq1}
\langle -x_k^*, x-x_{2k}\rangle\le g(x)-g(x_{2k})+\tilde\epsilon_k(\Vert x-x_{2k}\Vert+\vert g(x)-g(x_{2k})\vert)\; \forall x\in [a,b].
\end{equation}
Let $\lambda_{ik}\in (0,1]$ satisfy $x_{ik}:=\lambda_{ik}a+(1-\lambda_{ik}b)$ for $i=1,2.$ Then $b-x_{ik}=\lambda_{ik}(b-a)$ for $i=1,2.$ By taking $x=b$ into \eqref{thm33a-eq1}, we derive
\begin{align*}
\lambda_{2k}\langle x_k^*,b-a\rangle&\ge g(x_{2k})-g(b)-\tilde\epsilon_k(\lambda_{2k}\Vert b-a\Vert+\vert g(b)-g(x_{2k}\vert)\\
&=\lambda_{2k}\dfrac{f(b)-f(a)}{\Vert b-a\Vert}\Vert b-a\Vert -\tilde\epsilon_k\lambda_{2k}\left(\Vert b-a\Vert+\left\vert \dfrac{f(b)-f(a)}{\Vert b-a\Vert}\Vert b-a\Vert\right\vert\right)\\
&=\lambda_{2k}\left(f(b)-f(a)-\tilde\epsilon_k(\Vert b-a\Vert+\vert f(b)-f(a)\vert)\right),
\end{align*} which is equivalent to
\begin{align}\label{thm33a-eq2}
\langle x_k^*,b-a\rangle&\ge f(b)-f(a)-\tilde\epsilon_k(\Vert b-a\Vert+\vert f(b)-f(a)\vert).
\end{align}
This combines with the lsc and bounded property on $[a,b]$ of the function $\frac{\Vert b-\cdot\Vert}{\Vert b-a\Vert}$ to imply that
$$\liminf_{k\to\infty}\langle x_k^*,b-a\rangle\ge f(b)-f(a).$$ Thus \eqref{thm33a-eq00} holds true.

We now prove \eqref{thm33a-eq0}. Since $(b-x_{1k})\Vert b-a\Vert = \Vert b-x_{1k}\Vert(b-a),$ according to \eqref{thm33a-eq2}, we have
\begin{align*}
\liminf_{k\to\infty}\langle x^*, b-x_{1k}\rangle&=\liminf_{k\to\infty}\left(\dfrac{\Vert b-x_{1k}\Vert}{\Vert b-a\Vert}\langle x^*_k, b-a\rangle\right)\\
&=\liminf_{k\to\infty}\left(\dfrac{\Vert b-x_{1k}\Vert}{\Vert b-a\Vert}\left(f(b)-f(a)-\tilde\epsilon_k(\Vert b-a\Vert+\vert f(b)-f(a)\vert)\right)\right)\\
&=\dfrac{\Vert b-c\Vert}{\Vert b-a\Vert}(f(b)-f(a)).
\end{align*}
Thus,   inequality \eqref{thm33a-eq0} holds true with $x_{k}:=x_{1k}$ for all $k\in \N$. If $c\ne a,$ then we can assume that $\B(c,\sqrt{\epsilon_k})\cap [a,b]\subset (a,b)$ and that $x_{ik}=t_{ik}a+(1-t_{ik})b$, $t_{ik}\in (\sqrt{\epsilon_k},1-\sqrt{\epsilon_k})$ for $i=1,2.$ This implies that
\[
a-x_{ik}=(1-t_{ik})(a-b) \text{ \rm and } b-x_{ik}=t_{ik}(b-a) \text{ \rm for } i=1,2.
\]
Taking $x=a$ into \eqref{thm33a-eq1}, we have
\begin{align*}
&(1-t_{2k})\langle x_k^*,b-a\rangle=\langle -x_k^*,a-x_{2k}\rangle\\
&\le g(a)-g(x_{2k})+\tilde\epsilon_k(\Vert a-x_{2k}\Vert+\vert g(a)-g(x_{2k})\vert)\\
&= \dfrac{f(b)-f(a)}{\Vert b-a\Vert}\left(\Vert b-a\Vert-\Vert b-x_{2k}\Vert\right)(1+\tilde\epsilon_k) +\tilde\epsilon_k\Vert a-x_{2k}\Vert\\
&= \dfrac{f(b)-f(a)}{\Vert b-a\Vert}\left(\Vert b-a\Vert-t_{2k}\Vert b-a\Vert\right)(1+\tilde\epsilon_k)+\tilde\epsilon_k\Vert a-x_{2k}\Vert\\
& = (1-t_{2k})(f(b)-f(a))(1+\tilde\epsilon_k)+\tilde\epsilon_k\Vert a-x_{2k}\Vert.
\end{align*}
It follows that \begin{equation}\label{thm33a-eq02}
\limsup_{k\to\infty}\langle x_k^*, b-a\rangle \le f(b)-f(a).
\end{equation}
Take $x-b$ into \eqref{thm33a-eq1}, we derive that
\begin{align*}
&-t_{2k}\langle x_k^*,b-a\rangle=\langle -x_k^*,b-x_{2k}\rangle\\
&\le g(b)-g(x_{2k})+\tilde\epsilon_k(\Vert b-x_{2k}\Vert+\vert g(b)-g(x_{2k})\vert)\\
&= \dfrac{f(b)-f(a)}{\Vert b-a\Vert}\left(-\Vert b-x_{2k}\Vert\right)+\tilde\epsilon_k(\Vert b-x_{2k}\Vert+\vert g(b)-g(x_{2k})\vert)\\
&= \dfrac{f(b)-f(a)}{\Vert b-a\Vert}\left(-t_{2k}\Vert b-a\Vert\right)(1+\tilde\epsilon_k)+\tilde\epsilon_k\Vert b-x_{2k}\Vert\\
& = -t_{2k}(f(b)-f(a))(1+\tilde\epsilon_k) +\tilde\epsilon_k\Vert b-x_{2k}\Vert,
\end{align*} which implies that
$$
\langle x_k^*,b-a\rangle\ge (f(b)-f(a))(1+\tilde\epsilon_k)-\frac{\tilde\epsilon_k}{\sqrt{\epsilon_k}}\Vert b-x_{2k}\Vert.
$$
Passing to the infimum limit as $k\to\infty,$ we conclude $\liminf_{k\to\infty}\langle x_k^*,b-a\rangle\ge f(b)-f(a)$, which together with
  \eqref{thm33a-eq02}  obtain following equality
$$\lim_{k\to\infty}\langle x_k^*,b-a\rangle\ge f(b)-f(a),$$
which means that \eqref{thm33a-eq000} holds.
\end{proof}
\begin{cor}[Mean value inequality for Lipschitzian functions]\label{MVITheorem}
Let $f: X\to \bar{\R}$ be locally Lipschitz continuous relative to $[a,b]$ around any $x\in [a,b]$, where $a\in X$ and $b\in X.$  Then,
\begin{equation}\label{MVI-theorem-1}
\langle x^*, b-a\rangle\ge f(b)-f(a)\; \text{ \rm for some } x^*\in \partial_{[a,b]}f(c), \; c\in [a,b).
\end{equation}
\end{cor}

\begin{proof}
Using Theorem~\ref{thm33a}, we find $c\in [a,b)$ and sequences $\epsilon_k\to 0^+,$ $x_k\xrightarrow[f]{[a,b]} c$ and $x_k^*\in \widehat{\partial}^{\epsilon_k}_{[a,b]}f(x_k)$ satisfying
\begin{equation}\label{MVI-theorem-eq1}
\liminf_{k\to\infty}\langle x_k^*, b -a\rangle\ge f(b)-f(a).
\end{equation}
Since $f$ is locally Lipschitz continuous relative to $[a,b]$ around $x_k$ for all $k\in\N,$ we see that $\{x_k^*\}$ is bounded due to Lemma~\ref{lem31}~(ii). Without loss of generality, we can assume that $x_k^*\rightharpoonup^* x^*$ for some $x^*\in X^*$ and $\epsilon_k\downarrow 0,$ which implies that $x^*\in \partial_{[a,b]}f(\bar x).$ Taking \eqref{MVI-theorem-eq1} into account, we have
$\langle x^*, b-a\rangle\ge f(b)-f(a),$
which means that \eqref{MVI-theorem-1} holds. Hence, the proof is complete.
\end{proof}

Corollary~\ref{MVITheorem} is a generalized version of \cite[Corollary~3.51]{Mor06} as the following example shows.

\begin{exam}
Consider the indicator function $\delta_{[0,1]}.$ Then $\delta_{[0,1]}$ is locally Lipschitz continuous relative to $[a,b]$ around any $x\in [a,b].$ By Corollary~\ref{MVITheorem}, we find $c\in [0,1)$ and $x^*\in \partial_{[0,1]}\delta_{[0,1]}(c)$ such that
\begin{equation}\label{example-MVItheorem-eq1}
\langle x^*, b-a\rangle\ge \delta_{[0,1]}(1)-\delta_{[0,1]}(0)=0.
\end{equation}
By direct computation, we can see that for any $c\in (0,1)$ and $x^*\in \partial_{[0,1]}\delta_{[0,1]}(c)$, the inequality \eqref{example-MVItheorem-eq1} holds.

However, since $\delta_{[0,1]}$ is not locally Lipschitz continuous on any open set containing $[0,1],$ \cite[Corollary~3.51]{Mor06} is not applicable  for $\delta_{[0,1]}$ on $[0,1].$
\end{exam}

It is known by \cite[Proposition 2.107]{Bonnans2000} that if $f:X\to \bar \R$ is convex and continuous at $x\in X$, then $f$ is locally Lipschitz around $x.$ A similar result for convex functions on a set is presented in the following proposition.

\begin{pro}\label{pro5}
Let $f:X\to \bar \R$ be convex on $[a,b]$ and continuous relative to $[a,b]$ with $b\ne a$, where $a\in X$ and $b\in X$.  Then $f$ is locally Lipschitz relative to $[a,b]$ around any $x\in [a,b].$
\end{pro}

\begin{proof}
Fix $x\in [a,b]$. We find, from the continuity relative to $[a,b]$, $\delta>0$ such that
$$\vert f(y)-f(x)\vert\le 1 \text{ \rm for any } y\in\B(x,\delta)\cap [a,b],$$ which follows that
$\vert f(y)\vert\le \vert f(x)\vert+1=:M$  for all  $y\in \B(x,\delta)\cap [a,b].$ For any $u\ne v\in \B(x, \frac{\delta}{4})\cap [a,b],$ we denote $y:=\frac{\delta}{2}\Vert  v-u\Vert^{-1}(v-u)$ and $t:=\frac{2}{\delta}\Vert v-u\Vert.$ One has $t\in [0,1]$ and $v=u+ty$, so
$$f(u+ty)=f((1-t)u+t(y+u))\le (1-t)f(u)+tf(u+y)\le f(u)+t2M,$$ which in turn yields that
$f(v)-f(u)\le \frac{M}{\delta}\Vert v-u\Vert.$ As $u$ and $v$ are arbitrarily taken in $\B(x, \frac{\delta}{4})\cap [a,b],$ one has
$\vert f(v)-f(u)\vert\le \frac{M}{\delta}\Vert v-u\Vert,$ which means that $f$ is locally Lipschitz relative to $[a,b]$ around $x.$
\end{proof}

The next theorem provides relationships between the convexity on a set of a mapping with the monotonicity of its relative subdifferentials.

\begin{thm}[Subdifferential monotonicity and convexity on an interval of lsc functions] \label{thm34} Let $f:X\to\bar{\R}$ be proper and continuous relative to a closed interval $[a,b]\subset \dom f$, where $a\in X$ and $b\in X$. Then
the following statements are equivalent:

{\rm (i)} $f$ is convex on $[a,b]$;

{\rm (ii)} $\partial f_{[a,b]}: X\to X^*$ is monotone;

{\rm (iii)} $\partial f_{[a,b]}: X\to X^*$ is monotone and $f$ is locally Lipschitz relative to $[a,b]$ around any $x\in [a,b]$;

{\rm (iv)} $\partial_{[a,b]}f: X\to X^*$ is monotone and $f$ is locally Lipschitz relative to $[a,b]$ around any $x\in [a,b].$
\end{thm}

\begin{proof}  We first prove (i) $\Leftrightarrow$ (ii). Since $f$ is proper continuous relative to $[a,b]$ and convex on $[a,b],$ $f_{[a,b]}$ is proper, lsc, and convex, which is equivalent to the fact that $\partial f_{[a,b]}$ is monotone, due to \cite[Theorem~3.56]{Mor06}. Thus (i) $\Leftrightarrow$ (ii).

 From Proposition~\ref{pro5} and  fact  (i) $\Leftrightarrow$ (ii), one has (i) $\Leftrightarrow$ (iii). To complete the proof, one proves that (iii) $\Rightarrow$ (iv) $\Rightarrow$ (i).

Note that (iii) $\Rightarrow$ (iv) is obvious from the definition. It remains to prove (iv) $\Rightarrow$ (i). To see this, we need to prove that
\begin{equation}\label{thm34-eqa1}
\partial_{[a,b]}f(u)\subset\left\{u^*\in X^*\mid \langle u^*, x-u\rangle\le f(x)-f(u)\; \forall x\in [a,b]\right\}
\end{equation}
for any $u\in [a,b].$ Let $x,y\in [a,b]$ and $y^*\in \partial_{[a,b]}f(y).$ By Theorem~\ref{thm33a}, we find  $c\in [x,y)$ and sequences $\epsilon_k\to 0^+,$ $x_k\xrightarrow{[a,b]} c$ and $x_k^*\in \widehat{\partial}^{\epsilon_k}_{[a,b]}f(x_k)$ such that
\begin{equation}\label{thm34-eqa0a}
\frac{\Vert y-x\Vert}{\Vert y-c\Vert}\liminf_{k\to\infty}\langle x_k^*, y-x_k\rangle \ge f(y)-f(x).
\end{equation}
Since $f$ is locally Lipschitz relative to $[a,b]$ around $x_k,$ without loss of generality, we can assume that $x_k^*\rightharpoonup^* x^*$, which implies from \eqref{thm34-eqa0a} that $x^*\in \partial_{[a,b]}f(c)$ and
\begin{equation}\label{thm34-eqa0b}
\frac{\Vert y-x\Vert}{\Vert y-c\Vert}\langle x^*, y-c\rangle \ge f(y)-f(x).
\end{equation}
From the monotonicity of $\partial_{[a,b]}f$, we have $\langle x^*-y^*, c-y\rangle\ge 0,$  which is equivalent to
$\langle x^*, y-c\rangle\le \langle y^*,y-c\rangle.$ Taking  \eqref{thm34-eqa0b} into account and $\Vert y-c\Vert(y-x)=\Vert y-x\Vert(y-c)$, we obtain
$$\langle y^*, y-x\rangle=\frac{\Vert y-x\Vert}{\Vert y-c\Vert}\langle y^*, y-c\rangle\ge \frac{\Vert y-x\Vert}{\Vert y-c\Vert}\langle x^*, y-c\rangle \ge f(y)-f(x),$$
which yields that $\langle y^*, x-y\rangle \le f(x)-f(y)$ for all $x\in [a,b].$ Thus \eqref{thm34-eqa1} holds.

We continue to prove that $f$ is convex on $[a,b].$ Let $x,y\in [a,b]$ and $\lambda\in (0,1).$ Put $v:=\lambda x+(1-\lambda)y\in (a,b)$. Then $f$ is locally Lipschitz relative to $[a,b]$ around $v$.  In view of  Theorem~\ref{theo-non}, we can find $x^*\in\partial_{[a,b]}f(v)$. From \eqref{thm34-eqa1}, we have
$$f(y)-f(v)\ge\langle x^*, y-v\rangle \text{ \rm and } f(x)-f(v)\ge \langle x^*, x-v\rangle,
$$ so
$$(1-\lambda) f(y)+\lambda f(x)\ge (1-\lambda) f(v)+\langle x^*, (1-\lambda) y-(1-\lambda) v\rangle+\lambda f(v)+\langle x^*, \lambda x-\lambda v\rangle=f(v).$$
The proof of theorem is complete.
\end{proof}

\begin{remark}
Since the convexity of $f:X\to\bar\R$ on $[a,b]$  is equivalent to the monotonicity of $\partial f_{[a,b]}$, we see that the convexity of $f$ entails the monotonicity of $\partial_{[a,b]}f$. However, the opposite implication is another case because $\partial_{[a,b]}f(x)$ is actually contained in $\partial f_{[a,b]}(x)$ for $x\in X$. In many cases, $\partial_{[a,b]}f(x)$ for $x\in X$ is only a singleton, while $\partial f_{[a,b]}(x)$ could be a large set as we can see in the following  example.
\end{remark}

\begin{exam}\label{example34}
Consider the function $f$ defined by $f(x):=\frac{1}{3}x^3$ for $x\in\R.$ Then $f$ is convex on $[0,1].$ It is clearly that $f$ is locally Lipschitz relative to $[0,1]$ around any $x\in [0,1].$ Moreover, for any $x\in [0,1],$
$$\partial f_{[0,1]}(x):=\begin{cases}
(-\infty,0]&\text{ \rm if } x=0;\\
\{x^2\}& \text{ \rm if } x\in (0,1);\\
[1,\infty)& \text{ \rm if } x=1;\\
\emptyset& \text{ \rm otherwise},
\end{cases} \; \text{ \rm and }\;  \partial_{[0,1]}f(x):=\begin{cases}
\{x^2\}& \text{ \rm if } x\in [0,1);\\
\emptyset& \text{ \rm otherwise}.
\end{cases}$$
It is clear that $\gph\partial_{[0,1]}f\subsetneq \gph\partial f_{[0,1]}$ and   checking the monotonicity of $\partial_{[0,1]}f$ is much easier than   checking the monotonicity of $\partial f_{[0,1]}.$
\end{exam}

\section{Concluding Remarks}

In this paper, we  introduced new notions of relative subdifferentials for a function in reflexive Banach spaces. We proved that, under the locally Lipschitz continuity relative to  a set of a function around a reference point, the relative limiting subdifferential is not empty at this point.  We also provided the fuzzy sum rule for relative $\epsilon$-regular subdifferentials and the exact sum rule for the relative limiting subdifferentials. We then employed the relative subdifferentials  to derive  necessary optimality conditions for non-subdifferential sum optimization problems. Numerical examples were given to illustrate that the  optimality conditions obtained in this paper work better and sharper than some earlier results do.

With the help of the relative subdifferentials, we established new mean value theorems under very mild assumptions on the relative lsc and relative Lipschitz continuity and provided characterizations of the convexity on a set of a function through the monotonicity of the proposed relative subdifferentials. It would be interesting to see how the  relative subdifferentials can be applied to derive optimality principles for solving  wider classes of nonsmooth optimization problems including problems involving data uncertainties.


\begin{thebibliography}{99}
\bibitem{Zaorodny1990}
{\sc D.~Zaorodny}, {\em A note on the equivalence between the mean-value  theorem for the dini derivative and the clarke-rockafellar derivative},
  Optimization, 21 (1990),  179-183.
\end{thebibliography}
\end{document}